\documentclass[a4paper,11pt]{amsart}
\usepackage{amssymb, amsmath, mathrsfs}
\usepackage{mathtools}
\mathtoolsset{showonlyrefs}
\usepackage[inline, final]{showlabels}

\usepackage{tikz-cd}
\usepackage{mathrsfs}
\usepackage{enumitem}
\usepackage{shuffle}
\usepackage{pdfpages}
\usepackage{subcaption}
\usepackage[top=1in, bottom=1.25in, left=1.25in, right=1.25in]{geometry}
\usepackage{float}

\usepackage[]{todonotes}
\usepackage[
 pdfauthor={Yiannis N. Petridis  and Morten S. Risager},
 pdftitle={Counting and equidistribution over primes in hyperbolic groups},
 pdfsubject={Primary XXXX; Secondary  YYYY},
 pdfkeywords={},
 pdfproducer={LaTeX2e},
 pdfcreator={Pdflatex},
 pdfdisplaydoctitle =true,
 psdextra]
 {hyperref}
\usepackage{pdfsync}

\newtheorem{thm}{Theorem}[section]
\newtheorem{prop}[thm]{Proposition}
\newtheorem{lem}[thm]{Lemma}
\newtheorem{cor}[thm]{Corollary}
\newtheorem*{rem*}{Remark}

\theoremstyle{definition}

\newtheorem{example}[thm]{Example}

\newtheorem{rem}{Remark}

\def \e {{\varepsilon}}

\def \R {{\mathbb R}}
\def \H {{\mathbb H}}

\def \N {{\mathbb N}}
\def \Q {{\mathbb Q}}

\def \g {\gamma}
\def \G {\Gamma}
\def \Z {\mathbb{Z}}
\def \psl  {{\hbox{PSL}_2( {\mathbb R})} }
\def \sl  {{\hbox{SL}_2( {\mathbb R})} }

\def \slz  {{\hbox{SL}_2( {\mathbb Z})} }
\def \GmodH {{\Gamma\backslash\mathbb H}}

\def \Ginf {{\Gamma_{\!\infty}}}

\DeclareMathOperator{\arccosh}{arcosh}

\newcommand{\vol}[1]{\hbox{vol}\left( #1 \right)}
\newcommand{\abs}[1]{\left\lvert #1 \right\rvert}
\newcommand{\norm}[1]{\left\lVert #1 \right\rVert}

\DeclareMathOperator*{\li}{li}

\newcommand{\legendre}[2]{\ensuremath{\left( \frac{#1}{#2} \right) }}

\title{Counting and equidistribution over primes in hyperbolic groups}
\date{today}
\author[Y. Petridis]{Yiannis N. Petridis}
\address{Department of Mathematics, University College London, Gower Street, London WC1E 6BT, United Kingdom}
\email{i.petridis@ucl.ac.uk}

\author[M. Risager]{Morten S. Risager}
\address{Department of Mathematical
  Sciences, University of Copenhagen, Universitets\-parken 5, 2100
  Copenhagen \O, Denmark}
\email{risager@math.ku.dk}
 
\thanks{Morten S. Risager was supported by the Grant DFF-3103-00074B from Independent Research Fund Denmark. Both authors were supported by the Swedish Research Council under grant no. 2016-06596 while in residence at Institut Mittag-Leffler in Djursholm, Sweden during the Analytic Number Theory programme in the spring of 2024.}

\keywords{}
\subjclass[2020]{Primary 11J71, 
11N45 
; Secondary 11N36. 
 }
\date{\today}

\begin{document}
\begin{abstract} 

We consider equidistribution of angles for certain hyperbolic lattice points in the upper half-plane. Extending work of Friedlander and Iwaniec we show that for the full modular group equidistribution persists for matrices with $a^2+b^2+c^2+d^2=p$ with $p$ prime; at least if we assume sufficiently good lower bounds in the hyperbolic prime number theorem by Friedlander and Iwaniec. We also investigate  related questions for a specific arithmetic co-compact group and its double cosets by hyperbolic subgroups. The general equidistribution problem was studied by Good, and in this case, we show, that equidistribution holds unconditionally when restricting to primes.
\end{abstract}

\maketitle

\section{Introduction} Let $X$ be a compact topological space  equipped with a non-negative measure $\mu$, normalized such that $\mu(X)=1$.
It is an important problem in number theory to determine if a (generalised) sequence $S= (x_i)_{i\in I}\subseteq X$ is (asymptotically) equidistributed on $X$ (with respect to $\mu$), i.e. if for every $f\in C(X)$ we have that 
\begin{equation}
    \frac{1}{N_S(x)}\sum_{\nu(i)\leq x}f(x_i)\to \int_X f d \mu \textrm{ as }x\to\infty.
\end{equation}
Here we assume that $I$ is equipped with a (size) function $\nu:I\to \N$ with the property that $N_S(x):=\#\{i\in I:\nu(i)\leq x\}$ is finite for all $x\in \R$. 

It is an interesting question to see how much we can shrink the index set $I$ and still have (asymptotic) equidistribution. We say that a sequence $(x_i)_{i\in I}$ \emph{is (asymptotically) equidistributed on $X$ over primes} if  
for every $f\in C(X)$ we have that 
\begin{equation}
    \frac{1}{\pi_S(x)}\sum_{\substack{\nu(i)\leq x\\ \nu(i) \textrm{  prime}}}f(x_i)\to \int_X f d \mu \textrm{ as }x\to\infty.
\end{equation}
Here \begin{equation}\pi_S(x)=\#\{i\in I:\nu(i)\leq x, \nu(i) \textrm{ prime}\}.\end{equation} 
\begin{example}[Angles of lattice points in $\Z^2$]Consider the sequence of (normalized) angles of the square lattice $A=(arg(v)/2\pi)_{v\in \Z^2}\subseteq  \R\slash \Z$ with $\nu(v)=\norm{v}^2$. In the 19th century Gauss \cite{Gauss:2011} noticed that \begin{equation}\label{Euclidean-counting}N_A(x)=\pi x +O(x^{1/2}),\end{equation} (see e.g. \cite{BerndtKimZaharescu:2018} for a description of Gauss’ elementary method). His method is flexible enough to show that the normalized angles are equidistributed with respect to Lebesgue measure on $\R\slash \Z$. Landau \cite{Landau:1903} observed that $\pi_A(x)$ can be estimated, using the analytic properties of the Dedekind zeta-function $\zeta_K(s)$ for the number field $K=\Q(i)$;  in particular,  its pole at $s=1$ and the fact that it has a zero-free region. He found that 
\begin{equation}\label{Euclidean-counting-prime}\pi_A(x)=\sum_{p\leq x}r(p)=4 \hbox{li}(x)+O(x/\log(x)^A),
\end{equation}
where $\hbox{li}(x)=\int_{2}^x\frac{1}{\log t}dt$ is the logarithmic integral function. The factor 4 occurs since we are really counting prime ideal with multiplicity 4. In order to determine the angular distribution of these \lq prime lattice points\rq{}, Hecke \cite{Hecke:1920} introduced his Gr\"ossencharacters \begin{equation}\xi_k((z))=\left(\frac{z}{\abs{z}}\right)^{4k},\end{equation} proved analytic continuation of the corresponding Hecke $L$-functions $L(s, \xi_k)$ and proved, 
via Weyl’s equidistribution criterion, that $\arg(v)/2\pi$ equidistributes on primes.  See \cite[Eq (52)]{Hecke:1920}. See also the work of Kubilius \cite{Kubilius:1950} for a statement with error term.
\end{example}

\begin{example}[The sequence $\alpha n $ modulo 1 for $\alpha$ irrational] Using his equidistribution theorem, Weyl \cite[Satz 2]{Weyl:1916} proved that, if $\alpha\in \R\backslash \Q$, then $(\alpha n)_{n\in \N} \subseteq \R\slash \Z$ is equidistributed with respect to the Lebesgue measure on $\R\slash \Z$. Understanding what happens when restricting to primes is much harder, but Vinogradov famously proved that $(\alpha n)_{n\in \N}$ is indeed equidistributed over primes. See \cite[Thm 21.3]{IwaniecKowalski:2004a} for a proof. 

Analogous statements hold for $ q(n)$ with $q(x)$ a non-constant 
real polynomial with some condition on the coefficients (the leading coefficient being irrational suffices). For more on this consult \cite[Prop. 21.1]{IwaniecKowalski:2004a}, \cite[Thm 1]{Rhin:1973}.\end{example}

\subsection{Main results}
In this paper we investigate equidistribution over primes for various quantities in specific arithmetic subgroups of $\sl$. Anton Good \cite{Good:1983b} studied nine different types of decomposition of cofinite Fuchsian groups corresponding to double coset spaces 
$\G_\xi\backslash\G\slash \G_\chi$
for $\G_\xi, \G_\chi$  pairs of stabilisers of  $\xi$ resp. $\chi$. Here $\xi$, $\chi$ can be  cusps (parabolic subgroups), points in $\H$ (elliptic subgroups), or  geodesics between two points on the boundary of $\H$ (hyperbolic subgroups).  Using the spectral theory of automorphic forms he proved \cite[Cor. p 119]{Good:1983b} an equidistribution result for each of these nine types of decomposition. Good's results hold for any co-finite discrete subgroup of $\psl$, as their proof  utilize the spectral theory of the corresponding automorphic Laplacian, and not any arithmetic. Parkkonen and Paulin proved more general equidistribution results for endpoints of common perpendiculars in negative curvature, see \cite{ParkkonenPaulin:2017a}.

In this paper we consider arithmetic examples of his three diagonal cases $\xi=\chi$ and investigate what happens when we restrict the counting functions to primes. 

\subsubsection{The parabolic case} Consider $\G=\slz$, and the parabolic subgroup $\Ginf=\left\langle\begin{pmatrix}1&1\\0&1\end{pmatrix}\right\rangle$. We consider the sequence 
\begin{equation}
P=\left(\left(\frac{a}{c}, \frac{d}{c}\right)\right)\subseteq (\R\slash\Z)^2
\end{equation}
indexed over $\g \in  \Ginf\backslash\G\slash\Ginf$ with  $0<c$, and with size function $\nu(\g)=c$.

In this case Good's theorem (see also Selberg's unpublished notes \cite{Selberg:1977a}) implies that 
\begin{equation}
N_P(x)=\frac{3}{\pi^2}x^2+O(x^{4/3}),
\end{equation}
and $P \textrm{ is equidistributed on } (\R\slash \Z)^2$. 
Here and elsewhere equidistribution on $(\R\slash \Z)^2$ is understood to be with respect to Lebesgue measure.
For this specific group much better error terms are known using different methods (see \cite[p. 114 eq (2)]{Walfisz:1963a}, \cite{Liu:2016c}).

The sequence $P$ is straightforward to analyze over primes: Using the prime number theorem it follows that
\begin{equation}
    \pi_P(x)=\li(x^2)+O(x^2/\log^A(x)).
\end{equation}
Any non-trivial bound on the classical Kloosterman sums gives that 
$P$ is equidistributed on  $(\R\slash \Z)^2$ over primes.
We will not dwell on the details. The relevant decomposition  is a Bruhat type decomposition, see \cite[sec 1.4]{Iwaniec:2002a}.  Humphries \cite{Humphries:2022} proved several refined  equidistribution results in this case, using different notation.

\subsubsection{The elliptic case} Let $\G=\slz$, which acts on the upper half-plane $\H$. Consider the hyperbolic length $d_\H(\g i,i)$ between $\g i$ and  $i$, and the angle of the hyperbolic geodesic from $i$ to $\g i$ against the vertical geodesic from $i$. If $\g=\begin{pmatrix}a&b\\c&d\end{pmatrix}$ we let
\begin{equation}
    \nu_\H(\gamma)=a^2+b^2+c^2+d^2=2\cosh d_\H(\g i,i). 
\end{equation}
Selberg \cite{Selberg:1977a}, Nicholls \cite{Nicholls:1983a}, and Good \cite{Good:1983b} proved asymptotic equidistribution of angles related to $\g\in \G$ with $\nu_\H(\g)\leq x$. In fact Selberg and Good proved something even stronger. For $I\neq\g\in \sl$ the Cartan decomposition allows us to write \begin{equation}\g=k(\theta_1(\g))a(e^{-r})k(\theta_2(\g)),\end{equation} where 
\begin{equation}
k(\theta)=\begin{pmatrix}\cos\theta&\sin\theta\\
-\sin\theta&\cos\theta
\end{pmatrix}, \quad
a(e^{-r})=\begin{pmatrix}
    e^{-r/2}&\\&e^{r/2}
\end{pmatrix}.
\end{equation}
 Here $r=d_\H(\g i,i)>0$ is uniquely determined, and $\theta_1(\g), \theta_2(\g)$ are determined modulo $\pi$. The works of Selberg \cite{Selberg:1977a} and Good \cite{Good:1983b} imply that 
if we consider the sequence 
\begin{equation}\label{eq:E-sequence}E=(\theta(\g))_{\g\in\G}=((\theta_1(\g)/\pi, \theta_2(\g)/\pi))\end{equation} indexed over $\G$ and with counting function $\nu(\g)=\nu_\H(\g)$ then 
\begin{equation} \label{hyperbolic-counting}
    N_E(x)=6 x+O(x^{2/3}),
\end{equation}
and 
$E$  is equidistributed on $(\R\slash \Z)^2$. 
See also \cite{Patterson:1975a, Selberg:1977a, Good:1983b, Iwaniec:2002a}. 

In order to understand what happens when we restrict to primes, we recall that Friedlander and Iwaniec \cite{FriedlanderIwaniec:2009a} studied what they call a \lq hyperbolic prime number theorem\rq. In  our notation this means understanding $\pi_E(x)$.
Conditional on a conjecture weaker than the Elliott--Halberstam conjecture, see \eqref{hyperbolic-counting-prime-precise} below for the precise conjecture, they used sieving techniques to prove that 
\begin{equation} \label{hyperbolic-counting-prime}
\pi_E(x)
\asymp \frac{x}{\log x}.
\end{equation} Here $\asymp$ means that the quotient of the two sides is bounded from above and below by strictly positive constants. In this paper we show, conditional on the same conjecture, that $E$ is equidistributed over primes,
 analogously to Hecke's result for $\Z^2$. 

\begin{thm}\label{hyperbolic-equidistribution-prime} Let $E$ be the sequence \eqref{eq:E-sequence}, i.e. the pair of normalized angles from the Cartan decomposition for $\slz$. Assume $\pi_E(x)$ satisfies \eqref{hyperbolic-counting-prime}. Then $E$ is equidistributed on $(\R\slash \Z)^2$ over primes. 

\end{thm}

We remark that in Theorem \ref{hyperbolic-equidistribution-prime} we do not  need the full force of \eqref{hyperbolic-counting-prime}.
Our proof only requires
\begin{equation}\label{can this be proved unconditionally}
    \frac{\pi_E(x)}{x/\log(x)}(\log x)^{1-2\pi^{-1}}\to\infty 
\end{equation}
to conclude equidistribution over primes, which is much weaker than the conditional lower bound in \eqref{hyperbolic-counting-prime}. 
It would be interesting to see if \eqref{can this be proved unconditionally} can be proved unconditionally, and/or if the exponent of $1-2\pi^{-1}$ may be improved. 

Friedlander and Iwaniec also considered a related but different counting function namely 
\begin{equation}
    \pi_{E'} (x)=\#\{\g\in  \G\vert\, \nu_\H(\g)=p-2\leq x\}.
\end{equation}
This is equivalent to considering the sequence \begin{equation}\label{def:Emodified}E'=(\theta(\g))_{\g\in\G}=((\theta_1(\g)/\pi, \theta_2(\g)/\pi))\end{equation}  with modified  counting function $\nu'(\g)=\nu_\H(\g)+2$. Clearly we have $N_{E'}(x)=N_E(x)+O(x^{2/3})$ and $E'$ is also equidistributed on $(\R\slash \Z)^2$. 

For this slightly modified sequence they found precise unconditional asymptotics
\begin{equation}\label{hyperbolic-counting-prime-shifted}
    \pi_{E'}(x)=8\pi \prod_p\left(1+\frac{\chi_4(p)}{p(p-1)}\right)\hbox{li}(x)+O_A(x(\log x)^{-A})
\end{equation} for any $A>0$. Note that there is an obvious factor of 8 missing in going from (1.16) to (1.17) in \cite{FriedlanderIwaniec:2009a}.  Here $\chi_4$ is the primitive Dirichlet character modulo 4.
Our method is flexible enough to allow  considering the angle distribution in this case, and we arrive at the following unconditional result:

\begin{thm}\label{hyperbolic-equidistribution-prime minus 2} Let $E'$ be the modified sequence \eqref{def:Emodified}, i.e. the pair of normalized angles from the Cartan decomposition with the modified ordering. Then 
$E'$ is equidistributed on $(\R\slash \Z)^2$ over primes. 
\end{thm}

The hyperbolic prime number theorem \eqref{hyperbolic-counting-prime} of Friedlander and Iwaniec is hard. It is equivalent to the following statement:
\begin{equation}\label{conjecture-second version}
    \sum_{p\leq x}r(p+2)r(p-2)\asymp \frac{x}{\log x}.
\end{equation}
This highlights the similarity to the twin prime conjecture, and why \eqref{conjecture-second version} and  \eqref{hyperbolic-counting-prime} currently cannot be proved unconditionally. They are only proved conditionally on a weak form of the Elliott--Halberstam conjecture, see \eqref{hyperbolic-counting-prime-precise}. 

It turns out that, using Proposition \ref{prop:basic} below, Theorems \ref{hyperbolic-equidistribution-prime} and \ref{hyperbolic-equidistribution-prime minus 2} can be reformulated in terms of Gaussian integers. A Gaussian integer $w$ is called primary if $w=1 \bmod (1+i)^3$. Consider the following sets of pairs of primary Gaussian integers with norms differing by $4$:
\begin{equation}
 C_n=\left\{z_1\in \Z[i]\left\vert\, \begin{matrix}
    z_1\overline z_1 =n+2\\ 
     z_1\text{ primary}
 \end{matrix}\right.\right\}\times \left\{z_2\in \Z[i]\left\vert\, \begin{matrix}
    z_2\overline z_2=n-2\\
     \quad z_2\textrm{ primary}
 \end{matrix}\right.\right\}.  
\end{equation}
Note that for $n$ odd we have $\#C_n=r(n+2)r(n-2).$ For $z=(z_1,z_2)$ denote by $\theta(z)=(\theta(z_1),\theta(z_2))^t$ the corresponding set of angles.
Consider \begin{equation}
     M=\frac{1}{2\pi}\begin{pmatrix}
        1&\phantom{-}1\\1&-1
    \end{pmatrix},
\end{equation}
and observe that $M\theta(z)$ is the (normalized) sum and difference of the two angles. Consider the two sequences $\mathcal E$ resp. $\mathcal{E}'$ both defined to be the sum and the difference of the normalized angles of $z_1$ and $z_2 \bmod 1 $  i.e.
\begin{equation}
 \left(M\theta (z) \bmod \Z^2\right) \textrm{ indexed over }C_n,\textrm{ with }n\in \N,
\end{equation}
but with size function $\nu(z)=n$ resp. $\nu'(z)=n+2$. 
\begin{thm}\label{hyperbolic-equidistribution-prime-version-III}
The sequences $\mathcal{E}$ and $\mathcal{E}'$ are equidistributed  on $(\R\slash\Z)^2$. 
Assume $\pi_E(x)$ satisfies \eqref{conjecture-second version}, then $\mathcal{E}$ is equidistributed on $(\R\slash\Z)^2$ over primes. Unconditionally $\mathcal{E}'$ is equidistributed on $(\R\slash\Z)^2$ over primes.

\end{thm}
\begin{rem}
    For points \emph{on} hyperbolic circles Chatzakos, Kurlberg, Lester and Wigman \cite{ChatzakosKurlbergLesterWigman:2021} proved the existence of a full density subsequence of $n\in \mathcal N$ such that the angle $\theta_1(\g)$, for $\g\in \G$ with $\nu_\H(\g)=n$ equidistributes as   $n\to \infty$ with $n\in \mathcal N$. Here 
    \begin{equation}
    \mathcal N=\{n\in \N\vert n=a^2+b^2+c^2+d^2\textrm{ for some }\begin{pmatrix}a&b\\c&d\end{pmatrix}\in \Gamma \}.\end{equation} They also show that exceptional radii \emph{do} exist.  Similar questions for Euclidean circles were considered by K\'atai and K\"ornyei \cite{KataiKornyei:1976} and Erd\"os and Hall \cite{ErdosHall:1999}. Our method of proof is indeed inspired by \cite{ChatzakosKurlbergLesterWigman:2021}.

    Cherubini and Fazzari \cite{CherubiniFazzari:2022} extended \cite{ChatzakosKurlbergLesterWigman:2021} in a different direction by considering other CM-points with class number $h=1$. The techniques we are using probably generalise to congruence groups and other CM-points, but probably not to general cofinite subgroups of $\psl$.
\end{rem}    
   
\subsubsection{The hyperbolic case} The final case we analyze relates to the quaternion group 
\begin{equation}
\Gamma(2,5)= \left\{\begin{pmatrix}x_0+x_1\sqrt{2} & \sqrt{5}(x_2+x_3\sqrt{2})\\
\sqrt{5}(x_2-x_3\sqrt{2})&x_0-x_1\sqrt{2}
\end{pmatrix}\in \sl\vert x_i\in \mathbb Z\right\}
\end{equation} and its hyperbolic subgroup  
\begin{equation}H=\left\langle \begin{pmatrix}
    \varepsilon^2 &0\\ 0&\varepsilon^{-2}
\end{pmatrix}\right\rangle.
\end{equation} 
Here  $\varepsilon=1+\sqrt 2$ is the totally positive fundamental unit in the ring of integers of $\Q(\sqrt{2})$. 
 If $\g\in \G(2,5)$ has strictly positive integer entries, then there exist unique $y_1,y_2>0,v>0$ such that 
\begin{equation}\g=\pm \begin{pmatrix}\sqrt y_1&0\\0&1/\sqrt y_1\end{pmatrix}\begin{pmatrix}\cosh v&\sinh v\\ \sinh v&\cosh v \end{pmatrix}\begin{pmatrix}\sqrt y_2&0\\0&1/\sqrt y_2\end{pmatrix},
\end{equation}
 see Lemma \ref{full-hyperbolic-decomposition}.
Geometrically  $v(\g)$ equals half the distance between the infinite vertical geodesic $\mathcal{I}$ from $0$ to $i\infty$  and its image under $\g$. In terms of the entries one shows that $(\cosh(2 v(\g))-1)/10=bc/5$. 

Consider the sequence \begin{equation}\label{def:h}
h=(\psi(\g))=\left(\left(\frac{\log y_1}{2\log \e^2}, \frac{\log y_2}{2\log \e^2}\right)\right)\subseteq (\R\slash \Z)^2,\end{equation}
indexed over the set of all $\g\in H\backslash \G(2,5)\slash H$ with all four entries  strictly positive. Equip this index set with the size function $\nu(\g)=bc/5$.  Good \cite{Good:1983b} and Hejhal \cite[Thm. 8]{Hejhal:1982c} \cite{Hejhal:1978} proved that
\begin{equation}N_h(x)=\frac{10(\log\e)^2}{\pi^2}X+O(X^{2/3}).\end{equation} If there are non-zero eigenvalues of the automorphic Laplacian less than $1/4$, there are additional main terms, but these are not expected to exist in this case, as follows from Selberg's eigenvalue conjecture.
Good \cite[Thm. 4]{Good:1983b} further proved that 
$h$ is equidistributed on $(\R\slash \Z)^2$.
When restricting to primes we prove the following result: 
\begin{thm}\label{main-theorem-hyphyp} Consider the sequence $h$ in \eqref{def:h}. Then 
\begin{equation}\pi_h(x)=C\li(x)+O\left(\frac{x}{\log^Ax}\right),\end{equation}
and $h$ is equidistributed on $(\R\slash \Z)^2$ over primes.
Here   \begin{equation}
    C=\frac{12}{5}\frac{\log \e}{\sqrt{2}}\prod_{p\neq 5}\left(1+\frac{\chi_8(p)}{p(p-1)}\right),\end{equation} where $\chi_8$ is the even primitive Dirichlet character modulo $8$.
\end{thm}
\begin{rem}
    The sequence $h$ in \eqref{def:h} has a natural geometric interpretation: The geodesic segment minimizing the distance from the imaginary axis $\mathcal{I}$  and $\g \mathcal{I}$ meets $\mathcal{I}$ at $iy_1$ and $\g \mathcal{I}$ at $\gamma iy_2^{-1}$. The length of the geodesic corresponding to $H$ is $2\log (\e^2)$. The signed distance from $i$ to  $iy_1$ is $\log y_1$. So Good's theorem proves the (joint) equidistribution of the endpoints of the distance minimizing geodesic segments corresponding to the double cosets $H\backslash \G (2, 5) /H$.
\end{rem}

We can formulate the equidistribution statement entirely in terms of quantities defined using the ring of integers of the real quadratic field $K=\Q(\sqrt{2})$. Let $\sigma(a+\sqrt{2}b)=a-\sqrt{2}b$ be the non-trivial  Galois automorphism, and consider
\begin{equation}
    D_n=D_K(5n+1)\times D_K(n).
\end{equation}
Here $D_K(n)$ is the set of classes of totally positive elements of $\mathcal O_K$ with field norm $n$ modulo the following equivalence relation:  $z_1\sim z_2$ if and only if $z_1=\epsilon^{2m}z_2$ for some $m\in \Z$. Therefore,
\begin{equation}
    D_K(n)=\{z\in \mathcal O_K\vert z\cdot\sigma z=n, z>0, \sigma(z)>0\}\slash \sim.
\end{equation}
For $z\in(z_1,z_2)\in D_n$ denote 
\begin{equation}
\theta'(z)=\left(\frac{\log \abs{\frac{z_1}{\sigma z_1}}}{2\log \e^2}, \frac{\log \abs{\frac{z_2}{\sigma z_2}}}{2\log \e^2}\right),
\end{equation}
and let $M'=\frac{1}{2}\begin{pmatrix}1&1\\1&-1\end{pmatrix}$. Then we may consider the sequence $\mathcal H$ defined to be the (normalized) sum and difference of these quantities $\bmod\, 1$, i.e. 
\begin{equation}
(M'\theta'(z)\bmod \Z^2)\textrm{ indexed over }D_n \textrm{ with } n\in \N, 
\end{equation} and with size function $\nu(z)=n$. 
\begin{thm}
The sequence $\mathcal H$ is equidistributed on $(\R\slash\Z)^2$,    and equidistributed on $(\R\slash\Z)^2$  over primes. 
\end{thm}

\begin{rem}
In a nutshell the main theorems are proved by translating the statistics of the problem in question to a number field setting, where we can apply Hecke's theory of Gr\"ossencharacters, and extensions thereof. We then combine  with various bounds on sums involving multiplicative functions due to Nair and Tenenbaum, and use Weyl's equidistribution criterium.

In Section \ref{sec:bounds-multiplicative} we state the relevant various bounds on sums of multiplicative functions. In Section \ref{sec:elliptic} we translate the elliptic case to statistics of Gaussian integers on two circles, and analyze the relevant Weyl sums leading to Theorem \ref{hyperbolic-equidistribution-prime-version-III}. In Section \ref{sec:hyperbolic} we translate the hyperbolic case to statistics of the integers of the real quadratic field $\Q(\sqrt{2})$ on two hyperbolas. The counting problem can then be translated to an analogue of the Titchmarsh divisor problem, which we solve using recent results by Assing, Blomer and Li. The equidistribution result can be deduced using the techniques of Section \ref{sec:bounds-multiplicative}. This leads to a proof of Theorem \ref{main-theorem-hyphyp}.

\end{rem}
\section{Bounding sums of multiplicative functions} \label{sec:bounds-multiplicative}
A crucial ingredient in our proofs of equidistribution over primes are certain bounds on sums of multiplicative function. In this section we recall a few useful results in this direction.

Nair and Tennenbaum \cite{NairTenenbaum:1998} developed a general and very flexible method to harvest the power of multiplicativity  to bound specific sums of multiplicative or approximately multiplicative non-negative functions of various types. Here we state a small part of a simplified version of \cite[Thm 3]{NairTenenbaum:1998}:
\begin{thm}\label{nair-tenenbaum}Consider 
two non-negative multiplicative functions $g_1, g_2$ satisfying $g_i(n)\leq d(n)$ and $a_i,b_i$ satisfying $(a_i,b_i)=1$ and $b_1a_2\neq b_2a_1$. 
Then \begin{align}\label{nair-tennebaum}
    \sum_{p\leq x }g_1&(\abs{a_1p+b_1})g_2(\abs{a_2p+b_2})\\ & \ll_{a_i,b_i}\frac{x}{\log(x)}\prod_{2<p\leq x}\left(1-\frac{2}{p}\right)\sum_{n_1\leq x}\frac{g_1(n_1)}{n_1}\sum_{n_2\leq x}\frac{g_2(n_2)}{n_2}.
\end{align}
\end{thm}

It is well-known that
\begin{equation}\prod_{p\leq x}\left(1-\frac{2}{p}\right)=O((\log x)^{-2}),\end{equation}
e.g. it follows easily from the prime number theorem. Therefore, the right-hand side of \eqref{nair-tennebaum} is bounded by a constant times \begin{equation}\frac{x}{\log^3(x)}\sum_{n_1\leq x}\frac{g_1(n_1)}{n_1}\sum_{n_2\leq x}\frac{g_2(n_2)}{n_2}.\end{equation}

Another useful bound is the following weak form  of a Halberstam--Richert inequality \cite{HalberstamRichert:1979}. We quote from \cite[Thm 7.]{FainsilberKurlbergWennberg:2006}.
\begin{thm}\label{weak-halberstam-richert}
Let $f$ be a non-negative multiplicative function satisfying that \begin{equation}\displaystyle\sum_{n\leq x}f(n)=O(x),\textrm{ and }f(p^k)=O(k)\end{equation} for all primes $p$ and $k\geq 1$. Then 
\begin{equation}
\frac{1}{x}\sum_{n\leq x}f(n)\ll \exp\left(\sum_{p\leq x}\frac{f(p)-1}{p}\right)+\frac{1}{\log x}.
\end{equation}
The implied constant in the conclusion only depends on the implied constants of the assumptions.
\end{thm}

\section{The elliptic case}\label{sec:elliptic}

In this section we consider the modular group $\G=\slz$. We  recall the Cartan decomposition and its relation to angles. 
\subsection{Cartan decomposition and angles between lattice points}
For a group element $\g\in\sl=G$ the point $\g i\in \H$ is determined by the hyperbolic distance $d_\H(i, \g i )$ and the angle $\nu(\gamma)$ between the vertical geodesic from $i$ to $i\infty$ and the geodesic between  $i$ and $\g i$. To give a clear geometric picture we map the upper half-plane $\H$ to the Poincar\'e disc $\mathbb D$ using the Cayley map \begin{equation}f(z)=\frac{z-i}{z+i}.\end{equation} This is a holomorphic diffeomorphism with $f(i)=0$  Since it is conformal, it preserves angles. It maps the vertical geodesic from $i$ to the geodesic $[0,1)$ in the Poincar\'e disc, so that $\nu(\gamma)$ is the argument of the complex number $f(\g i)$, i.e. \begin{equation}f(\g i)=\abs{f(\g i)}e^{i\nu(\g)}.\end{equation}
    By the Cartan decomposition $G=KAK$  we may write 
\begin{equation}\g=k(\theta_1(\g))a(e^{-r})k(\theta_2(\g)),\end{equation}
where 
\begin{equation}
k(\theta)=\begin{pmatrix}\cos \theta &\sin\theta \\
-\sin \theta &\cos \theta
\end{pmatrix}\in K=\hbox{SO}_2(\mathbb{\R}),
\end{equation}
which is the stabiliser of $i$ in $\sl$, and 
\begin{equation}
a(e^{-r})=\begin{pmatrix}
    e^{-r/2}&\\&e^{r/2}
\end{pmatrix}\in A=\left\{\begin{pmatrix}a&\\&a^{-1}\end{pmatrix}\vert\, a>0\right\}.
\end{equation}
Here $r=d_\H(\g i,i)\geq 0$ is uniquely determined and, if $r>0$, the numbers $\theta_1(\g), \theta_2(\g)$ are determined modulo $\pi$, while $\theta_1(\g)+\theta_2(\g)$ is determined modulo $2\pi$.
A straightforward computation shows that \begin{equation}
    f(\g i)=\frac{e^{-r}-1}{e^{-r}+1}e^{2i\theta_1(\g)},
\end{equation}
so $2\theta_1(\g)=\nu(\g)$. We note also that 
\begin{equation}\g^{-1}=k(\pi/2-\theta_2(\g))a(e^{-r})k(-\pi/2-\theta_1(\g)),
\end{equation} so $2\theta_2(\g)=\pi-2\theta_1(\g^{-1})=\pi-\nu(\g^{-1})$.
When studying the joint distribution of $\theta_1(\g),\theta_2(\g) \bmod \pi$ we want to consider the corresponding Weyl sums which in this case are
\begin{equation}\label{def:S_e}
S_e(m_1,m_2,n)=\sum_{\substack{\g\in\G\\ \nu_\H(\g)=n}}\exp({i(2\theta_1(\g)m_1+2\theta_2(\g)m_2)}).
\end{equation}
This may be thought of as a Kloosterman type sum related to the Cartan decomposition, in the same way that the standard Kloosterman sum is related to a Bruhat type decomposition $G=NAN\cup N\omega AN$, where $N$ consists of upper triangular matrices with 1 on the diagonal and $\omega=k(-\pi/2)$.

It is convenient to mod out on the right and left by of stabiliser $\G_i$ of $i$ in $\G$ acting on $\H$. One finds that $\G_i$ is cyclic of order $4$ generated by the elliptic element \begin{equation}
    \g_i=\begin{pmatrix}
    0 &1\\-1&0
\end{pmatrix}=k(\pi/2),
\end{equation}  so that \begin{equation}\label{eq:double-representation}\g_i^{j_1}\g
\g_i^{j_2}=k(\theta_1(\g)+j_1\pi/2)a(e^{-r})k(\theta_2(\g)+j_2\pi/2).\end{equation} Since if $\g\neq \pm I$, the double coset $\G_i\g \G_i$ contains precisely $8$ elements; these can e.g. be parametrised by $\g_i^{j_1}\g\g_i^{j_2}$ for $j_1=0,\ldots ,3$, and $j_2=0,1$. Another possible parametrisation is $j_1=0,1$, $j_2=0,\ldots ,3$. 
Fixing representatives for the double coset $\G_i\backslash \G\slash \G_i$ and using \eqref{eq:double-representation}, we find that 
\begin{align}
\label{eq:reduce to quotient}    S_e(m_1,m_2,n)=
    &\delta_{2\mid m_1}\delta_{2\mid m_2}8 S'_e(m_1,m_2,n),
\end{align}
where  
\begin{equation}
S'_e(m_1,m_2,n)=\sum_{\substack{\g\in\G_i\backslash\G\slash\G_i\\ \nu_\H(\g)=n}}\exp({i(2\theta_1(\g)m_1+2\theta_2(\g)m_2)}).
\end{equation}
Note that the relation \eqref{eq:reduce to quotient} shows that we only need to bound the Kloosterman-type sum $S'_e(m_1,m_2,n)$ when $m_1,m_2$ are even. It also shows that in this case $S'_e(m_1,m_2,n)$ is independent of the choice of representatives for the double coset.

\subsection{The Gaussian integers}
It is a remarkable fact, observed in part in \cite{FriedlanderIwaniec:2009a} and \cite{ChatzakosKurlbergLesterWigman:2021} that for $\G=\slz$ the data $\{(\nu_\H(\g), \theta_1(\g), \theta_2(\g))\vert\, \g \in \G\}$ can be parametrised in terms of angles and lengths of Gaussian integers on two Euclidean circles with distance $4$ apart. 

Before we explain this observation in detail, we recall some results about Gaussian integers. For more information the reader may consult \cite{IrelandRosen:1990a, IwaniecKowalski:2004a}.

The Gaussian integers $\Z[i]$ is the ring of integers of the imaginary quadratic field $\Q(i)$. The field is equipped with two important multiplicative maps; complex conjugation $z\mapsto \overline z$, and the norm map $N(z)=z\overline z$. It is a Euclidean domain with respect to this norm. The group of units satisfies \begin{equation}\Z[i]^\times=\{z\in \Z[i]\vert\, N(z)=1\}=\{\pm 1, \pm i\}.\end{equation} The ring of integers is a unique factorisation domain and the irreducible elements are, up to multiplication by a unit, 
\begin{enumerate}[label=\roman*)]
\item  $(1+i)$,
\item $\pi_p, \bar\pi_p, \textrm{where }\pi_p=x+iy\textrm{ with   }x^2+y^2=p\equiv 1\bmod 4$, 
\item $p=3 \bmod 4$. 
\end{enumerate}

A Gaussian integer $\alpha\in \Z[i]$ is called \emph{primary}  if $\alpha=1 \bmod (1+i)^3$. Note that with this definition the only primary unit is 1. For  $\alpha$ not divisible by $(1+i)$, there exists a unique unit $u$ such that $u\alpha$ is primary. Every primary element can be written uniquely as a product of primary irreducible elements. A primary element $z$  satisfies $N(z)=1\bmod 4$. 

We fix a specific set of irreducible elements by fixing $(1+i)$ and for the other irreducible elements we choose the primary irreducible. Note that $\pi$ is primary if and only if $\overline \pi$ is primary. 
For any of these specific irreducible elements $\pi$ we write \begin{equation}
    \pi=\abs{\pi}e^{i\theta_\pi}.
\end{equation}
Note that, if $p=3\bmod 4$, then the corresponding primary irreducible is $\pi=-p$, and $\theta_\pi=\pi$. If, on the other hand, $p=1\bmod 4$ and $p=\pi\overline{\pi}$ is the factorisation into primary irreducibles we denote $\theta_\pi=\theta_p$ , $\theta_{\overline{\pi}}=-\theta_p$. To disambiguate the choice of $\pi$ vs $\overline \pi$ we may assume $\theta_p>0$. 

The angles $\theta_\pi$ of the primary irreducibles are asymptotically equidistributed modulo $2\pi$.
This can be seen by using Weyl's equidistribution criterion. To see why this applies one considers the set of primitive Hecke Gr\" ossencharacters described in \cite[Ex 1. p. 62]{IwaniecKowalski:2004a}, 
the analytic properties of the corresponding Hecke $L$-function  $L(s, \xi_k)$,see \cite[Thm 3.8]{IwaniecKowalski:2004a},
and a standard zero-free region for these functions. 

\subsection{Analysis on Weyl sums in \texorpdfstring{$\mathbb{Z}[i]$}{Z[i]}}
We now describe certain Weyl sums related the Gaussian integers.
Consider
\begin{align}
    W_m(n)&=\frac{1}{4}\sum_{\substack{z\in \Z[i]\\N(z)=n}}\left(\frac{z}{\abs{z}}\right)^m, &&\textrm{for $n\in \N$,}\\
W^P_m(n)&=W^P_m(2^l)W^P_m(n'),&&\textrm{if $n=2^ln'$ with $n'$ odd,}\\
\intertext{where}
W^P_m(n)&=\sum_{\substack{z\in \Z[i]\\N(z)=n\\ z\text{ primary}}}\left(\frac{z}{\abs{z}}\right)^m&&\textrm{ for $n$ odd.}\\
    W^P_m(2^l)&=\left(\frac{1+i}{\sqrt{2}}\right)^{lm}=e^{i\frac{\pi}{4}lm}.
    \end{align}
Note that for $n$ odd \begin{equation}W^P_m(n)=\sum_{\substack{z\in \Z[i]\slash \Z[i]^\times\\N(z)=n}}\left(\frac{z}{\abs{z}}\right)^m\end{equation}
for a specific choice of representatives of $\Z[i]\slash \Z[i]^\times$.
 It follows that 
\begin{align}
    W_m(n)&=\frac{1}{4}\sum_{u\in \Z[i]^\times}u^{m}W^P_m(n) =\begin{cases}W^P_m(n) ,&\textrm{ if }m=0\bmod 4\\
    0,& \textrm{otherwise.}
    \end{cases}
\end{align}

The functions $W_m(n)$ and $W^P_{m}(n)$ are multiplicative, i.e.
\begin{align}
    W_m(n_1n_2)&=W_m(n_2)W_m(n_2) ,\\
    W^P_m(n_1n_2)&=W^P_m(n_1)W_m(n_2)  ,  
\end{align}
if $(n_1,n_2)=1$.
The statement for $W_m$ follows from unique factorisation into irreducible elements in $\Z[i]$ up to associates, and the statement for $W_m^P$ follows from unique factorisation into primary irreducibles. Furthermore we have the trivial bound 
\begin{equation}\label{a trivial-bound}
    \abs{W^P_m(n)}\leq W^P_0(n)=r(n)/4\leq d(n)=O_\e(n^\e), \textrm { as }n\to\infty.
\end{equation}

We also note that, if $p=3\bmod 4$, then  
\begin{equation}\label{eq: W^P(p^l):p =3 mod 4}
    W^P_m(p^l)=\begin{cases}
    0,&\textrm{if $l$ is odd},\\
    1,&\textrm{if $l$ is even,}
    \end{cases}
\end{equation}
and, if $p=1\bmod 4$, then 
\begin{equation}
    W^P_m(p^l)=\sum_{j=0}^le^{i(2j-k)m\theta_p}.
\end{equation}
In particular we have for any odd rational prime
\begin{equation}
W^P_m(p)=\begin{cases}
2\cos(m\theta_p),&\textrm{if }p=1\bmod 4,\\
0, &\textrm{if }p=3\bmod 4.\\
\end{cases}
\end{equation}
For $m=0$, the result \eqref{Euclidean-counting} of Gauss gives 
\begin{equation}
\sum_{n\leq x} W^P_0(n)=\frac{\pi}{4}x+O(x^{1/2}).
\end{equation}

\begin{prop}\label{Average-Euclidean Weyl sum bounded} Let $m$ be an even non-zero integer. Then
\begin{equation}
    \sum_{n\leq x}\abs{W^P_m(n)}=O\left(x\left(\frac{\log^2\abs{m}}{\log x}\right)^{1 -2/\pi}\right).
\end{equation}
\end{prop}
\begin{proof} 

When $m=0\bmod 4$ we have $W_m(n)=W_m^P(n)$ and in this case the claim is \cite[Prop. 6]{FainsilberKurlbergWennberg:2006}. 

To handle the general situation we note the following. For every even $m$ we have that, if $z_j\in \Z[i]$,  $j=1,2$ are associated, and have angles determined by $z_j=\abs{z}e^{i\theta(z_j)}$, then there exists an  $r=0,\ldots,3$ such that $\theta(z_1)=\theta (z_2)+r\pi/2 \bmod 2\pi$. It follows that $\abs{\cos(m\theta(z_1))}=\abs{\cos(m\theta(z_2))}$. In particular, if $\nu_p=\arctan(y/x)$, where $p=x^2+y^2=1\bmod 4$ with $0\leq y\leq x$, then  we have \begin{equation}\abs{\cos(\theta_p)}=\abs{\cos(m\nu_p)}.\end{equation}
With this observation, see also  \cite[p. 2367]{ChatzakosKurlbergLesterWigman:2021} we can repeat the argument in \cite[p.91--92] {ErdosHall:1999}  for every even non-zero $m$ and find 
\begin{equation}
\sum_{\substack{p\leq x\\p=1\bmod 4}}\frac{\abs{\cos(m\theta_p)}}{p}\leq \frac{1}{\pi}\log\log x +(1-2/\pi)\log\log m+O(1),
\end{equation}
when $\log m\leq b\sqrt{\log x}$. Notice that our $\nu_p$ is $\theta(p)$ in \cite{ErdosHall:1999}. Once this has been established the proof in \cite[Prop. 6]{FainsilberKurlbergWennberg:2006} carries through verbatim and gives the result. 
\end{proof}

\subsection{Parametrisation of \texorpdfstring{$\G$}{G} in terms of Gaussian integers}
We can now explain a parametrisation of the elements of $\G$ in terms of Gaussian integers with various norms.  The basic map \eqref{eq:basic-map} below is a variation of the one in  \cite{FriedlanderIwaniec:2009a},\cite{ChatzakosKurlbergLesterWigman:2021}.

For $\g=\begin{pmatrix}
    a&b\\c&d
\end{pmatrix}\in \G$ we define two Gaussian integers
\begin{align}
\begin{split}
    \label{eq:basic-map}
    z_1(\g)=(a+d)+i(b-c)\in \Z[i],\\
    z_2(\g)=(a-d)-i(b+c)\in \Z[i] .
\end{split}
\end{align}
It is straightforward to verify the following proposition whose proof we leave as an exercise. 
\begin{prop}\label{prop:basic}\phantom{123}
\begin{enumerate}
    \item \label{prop:basic-injective}The map $\g\mapsto (z_1(\g),z_2(\g)) $ is injective. 
    \item \label{prop:basic-length}If $\nu_\H(\g)=n$, then $N(z_1(\g))=n+2$ and $N(z_2(\g))=n-2$.
    \item \label{prop:basic-quotients}If $\g'=\g_i^{j_1}\g\g_i^{j_2}$, then 
    \begin{align}
            z_1(\g')&=i^{j_1+j_2}z_1(\g),\\
            z_2(\g')&=i^{j_1-j_2}z_2(\g).
    \end{align}
   
    \item \label{prop:basic-disc-model}We have
    \begin{align}
        f(\g i)&=\frac{z_2(\g)}{\overline{z_1(\g)}}=z_2(\g)/\overline{z_1(\g)}, \\f(\g^{-1} i)&=-\frac{z_2(\g)}{z_1(\g)}=-z_2(\g)/z_1(\g).
    \end{align}
    \item \label{prop:basic-angles}We have
     \begin{align}
        \exp{(2 i\theta_1(\g))}&=\frac{z_1(\g)}{\abs{z_1(\g)}}\frac{z_2(\g)}{\abs{z_2(\g)}},\\ 
        \exp{(2 i\theta_2(\g))}&=\frac{z_1(\g)}{\abs{z_1(\g)}}\frac{\abs{z_2(\g)}}{z_2(\g)}.
    \end{align}
\end{enumerate}
\end{prop}
We note that Proposition \ref{prop:basic} \eqref{prop:basic-angles} can be formulated as follows: the angles $2\theta_1(\g), 2\theta_2(\g)$ can be identified as the sum and difference of the arguments of $z_1(\g), z_1(\g)$, i.e.
\begin{align}
    2\theta_1(\g)&=\arg(z_1(\g))+\arg(z_2(\g)) \bmod 2\pi,\\
    2\theta_2(\g)&=\arg(z_1(\g))-\arg(z_2(\g)) \bmod 2\pi.
\end{align}

Proposition \ref{prop:basic} \eqref{prop:basic-injective}, \eqref{prop:basic-length} raises the question of determining the precise image of the map $\g\mapsto (z_1(\g),z_2(\g)) $ when restricted to the finite set consisting of $\g\in\G$ with $\nu_\H(\g)=n$. Since an invertible $2\times 2$ matrix must have at least two non-zero entries consider $n\geq 2$. We consider this question separately for the four different values of $n\bmod 4$:

Let $S_2=\{n\in \Z \vert\, r(n)>0\}$ be the set of integers expressible as a sum of two squares. Recall that  $S_2$ consists precisely of integers $n$ satisfying that any prime $p=3\bmod 4$ occurs an even number of times in the factorisation of $n$ into rational primes, so an odd number cannot be in $S_2$ unless it is $1\bmod 4$.
\begin{lem}\label{translation: Trivial case} If $n=0\bmod 4$ or $n=1\bmod 4$ then
\begin{equation}
    \{\g\in \G\vert\, \nu_\H(\g)=n\}=\emptyset.
\end{equation}
\end{lem}

\begin{proof}

Case $n=0\bmod 4$: Proposition \ref{prop:basic} \eqref{prop:basic-length} implies that if $\{\g\in \G\vert\, \nu_\H(\g)=n\}\neq \emptyset$ then $n\pm 2\in S_2$. We have $n=4m$, so $n\pm 2=2(2m\pm 1)$. Clearly one of  $2m\pm 1$ is equal to $3\bmod 4$, so one of them is not in $S_2$. This implies that the prime factorisation of that number contains a prime $p=3\bmod 4$ occurring an odd number of times. But then the same is true for the corresponding $n\pm 2=2(2m\pm 1)$, so one of $n\pm 2$ is not in $S_2$. This implies that in this case  
\begin{equation}
    \{\g\in \G\vert\, \nu_\H(\g)=n\}=\emptyset.
\end{equation}

Case $n=1\bmod 4$:
Again Proposition \ref{prop:basic} \eqref{prop:basic-length} implies that, if $\{\g\in \G\vert\, \nu_\H(\g)=n\}\neq \emptyset$, then $n\pm 2\in S_2$. When $n=1 \bmod 4$ we have that $n+2\notin S_2$, so also in this case 
\begin{equation}
    \{\g\in \G\vert\, \nu_\H(\g)=n\}=\emptyset.
\end{equation}
\end{proof}
Let 
\begin{equation}B_n=\{z_1\in \Z[i]\vert\, N(z_1)=n+2\}\times \{z_2\in \Z[i]\vert\, N(z_2)=n-2\}.\end{equation}
\begin{lem}\label{lem:translation: first-case}
Let $n\geq 2$ with $n=2\bmod 4$. Then the map 
\begin{equation}
    \begin{array}{ccc}
    \{\g\in \G\vert\, \nu_\H(\g)=n\}&\to & B_n\\
    \g&\mapsto & (z_1(\g), z_2(\g))
    \end{array}
\end{equation}
is an isomorphism. 
\end{lem}
\begin{proof}

We can define an inverse map from $B_n$. Proposition \ref{prop:basic}  gives that the assignment $\g\mapsto (z_1(\g), z_2(\g))$ maps $\{\g\in \G\vert\, \nu_\H(\g)=n\}$ into $B_n$, so we have an injective map into $B_n$. 

To see that it is surjective we note that if $(z_1,z_2)\in B_n$ with $z_j=x_j+iy_j$, then \begin{equation}N(z_1)=n+2= 0\bmod 4, \quad N(z_2)=n-2= 0\bmod 4.\end{equation} 
But this is only possible if all of $x_1,x_2, y_1,y_2$ are even. To see this write \begin{align}x_1=\delta_x+2m_x\\ y_1=\delta_y+2m_y\end{align} with $m_x, m_y$ integers and $\delta_x, \delta_y\in \{0,1\}$. Then 
\begin{equation}
x_1^2+y_1^2=\delta_x^2+4\delta_xm_x+4m_x^2+\delta_y^2+4\delta_ym_y+4m_y^2=\delta_x^2+\delta_y^2\bmod 4,
\end{equation} which  implies $\delta_x=\delta_y=0$,  since $N(z_1)=0 \bmod 4$.

We now define a map $B_n\to \{\g\in \G\vert\, \nu_\H(\g)=n\}$ as follows: Let
\begin{equation}\label{eq:inverse-map}
a=\frac{x_1+x_2}{2},\quad b=\frac{y_1-y_2}{2},\quad c=\frac{-y_1-y_2}{2}, \quad  d=\frac{x_1-x_2}{2}.
\end{equation}
Since all the coordinates of $z_1,z_2$ are even, $a,b,c,d$ are all integers, and we easily find $ad-bc=1$ and $a^2+b^2+c^2+d^2=n$. Setting $$\gamma(z_1,z_2):=\begin{pmatrix}
    a&b\\c&d
\end{pmatrix}\in \G,$$ we verify that $z_i(\gamma(z_1,z_2))=z_i$, i.e. we have constructed an inverse to $\g\mapsto (z_1(\g),z_2(\g))$. 

Finally, we note that both sides of the map might be empty, e.g.  if $n=10$ since $12\notin S_2$. 

\end{proof}

Recall that we have defined \begin{equation}
  C_n=\left\{z_1\in \Z[i]\left\vert\, \begin{matrix}
    N(z_1)=n+2\\ 
     z_1\text{ primary}
 \end{matrix}\right.\right\}\times \left\{z_2\in \Z[i]\left\vert\, \begin{matrix}
    N(z_2)=n-2\\
     \quad  z_2\textrm{ primary}
 \end{matrix}\right.\right\}.
\end{equation}

\begin{lem}\label{lem:translation: second-case}
Let $n\geq 2$ with $n=3\bmod 4$. Then there exist unique representatives for the double cosets in $\G_i\backslash\G\slash\G_i$ with $\nu_\H(\g)=n$ such that
\begin{equation}
    \begin{array}{ccc}
    \{\g\in \G_i\backslash\G\slash\G_i\vert\, \nu_\H(\g)=n\}&\to & C_n\\
    \g&\mapsto & (z_1(\g), z_2(\g))
    \end{array}
\end{equation}
is an isomorphism. 
\end{lem}

\begin{proof}

If a Gaussian integer is primary then its imaginary part is even and its real part is odd \cite[p 121]{IrelandRosen:1990a}. With these parity conditions the numbers $a,b,c,d$ defined by \eqref{eq:inverse-map}
are again all integers and we have $ad-bc=1$ and $a^2+b^2+c^2+d^2=n$. Setting $\gamma(z_1,z_2):=\begin{pmatrix}
    a&b\\c&d
\end{pmatrix}\in \G$ we again easily verify that $z_i(\gamma(z_1,z_2))=z_i$. However, the map from $\{\g\in \G\vert\, \nu_\H(\g)=n\}$ does not always land in $C_n$ so we need to restrict our mapping a bit.

We claim that any double cosets in $\G_i\backslash\G\slash\G_i$ has a unique representative $\g$ such that $(z_1(\g), z_2(\g))\in C_n$. To see this we first take any representative $\g''$ for a given double coset. Then $N(z_2(\g''))=1\bmod 4$ so $z_2(\g'')$ is not divisible by $(1+i)$. 
It follows that there exists a unique $k\bmod 4$ such that $i^kz_2(\g'')$ is primary. Let $\g'=\g_i^{k}\g''$.  Then by Proposition \ref{prop:basic} \eqref{prop:basic-quotients} we have that   $z_2(\g')=i^kz_2(\g')$, so \begin{equation}z_2(\g')\in \left\{z_2\in \Z[i]\left\vert\, \begin{matrix}
    N(z_2)=n-2\\
     z_2\text{ primary}
 \end{matrix}\right.\right\}.\end{equation} We now claim that $z_1(\g')=\pm 1 \bmod (1+i)^3$.

To see this we first note that for any $w\in \Z[i]$ is congruent to 0 or 1 modulo $(1+i)$. This follows from observing that Euclidean division in $\Z[i]$ gives that $w=q(1+i)+r$ for some $q\in \Z[i]$ where $r$ is either zero or a unit. Noting that all units are equivalent modulo $(1+i)$ shows that $w=0,1\bmod (1+i)$.

To prove that $z_1(\g')=\pm 1 \bmod (1+i)^3$ we note that, by construction, we have that $z_2(\g')=1\bmod (1+i)^3$, and from the general definition \eqref{eq:basic-map} we see that 
$z_1(\g')=z_2(\g')+2w$, where $w=a+ib$. Since $2w$ equals $0$ or $2$ $\bmod (1+i)^3$, it follows that $z_1(\g)=\pm 1\bmod (1+i)^3$.

If $z_1(\g')=1\bmod (1+i)^3$, then the only other representative in the same double coset  that map to the same point is  $\g=\g_i^2\g'\g_i^2$, see  again Proposition \ref{prop:basic} \eqref{prop:basic-quotients}. Since $\g_i^2=-I$ this is really the same representative, i.e. $\g=\g_i^2\g'\g_i^2=\g'$.  

If $z_1(\g')=-1\bmod (1+i)^3$, then the only representatives in the same double coset  that maps to $C_n$ is  $\g=\g_i^1\g'\g_i^1$ or $\g=\g_i^3\g'\g_i^3$. But since $\g_i^2=-I$ this is again really the same representative. This finishes the proof.

Note that also in this case we may have $\{\g\in \G\vert\, \nu_\H(\g)=n\}=\emptyset$, e.g. when $n=19$, since $21\notin S_2$. 

\end{proof}

We can now use the bijections of the above lemmata to connect the hyperbolic Kloosterman sum with products of Weyl sums related to Gaussian integers with norms with distance $4$ apart:
\begin{lem} \label{lem:factorisation} If $m_1$ or $m_2$ is odd, then $S_e(m_1,m_2,n)=0$. If not, then
\begin{equation}
 S_e(m_1,m_2,n)=\begin{cases}
  16W_{m_1-m_2}(n+2)W_{m_1+m_2}(n-2),  &\textrm{if }n=2\bmod 4,\\
8W^P_{m_1-m_2}(n+2)W^P_{m_1+m_2}(n-2),&\textrm{if }n=3\bmod 4 ,\\
    0,&\textrm{otherwise.}
    \end{cases}
\end{equation}
\end{lem}
\begin{proof}
For $n=0,1 \bmod 4$ this follows from Lemma \ref{translation: Trivial case}. For $n=2 \bmod 4$ it follows from Lemma \ref{lem:translation: first-case} and Proposition \ref{prop:basic} \eqref{prop:basic-angles}.
For $n=3\bmod 4$ it follows from \eqref{eq:reduce to quotient}, Lemma \ref{lem:translation: second-case}, and Proposition \ref{prop:basic} \eqref{prop:basic-angles}.
\end{proof}
\subsection{Counting group elements along primes in the elliptic case}
Friedlander and Iwaniec considered
\begin{equation}
    \pi_\G(x)=\#\{\g\in  \G\vert\,\, \nu_\H(\g)=p\leq x\}.
\end{equation} We note that $\pi_\G(x)=\pi_E(x)$.

 In order to give good estimates on $\pi_E(x)$ they introduced for $0<\theta\leq 1$ an assumption $A(\theta)$ as follows. 
Let $\Lambda(n)$ denote   
the von Mangoldt function and \begin{equation}
\psi(x,a,q)=\sum_{\substack{n\leq x\\n= a\bmod q}}\Lambda(n),
\end{equation}
which by the prime number theorem for primes in arithmetic progression is asymptotic to $x/\varphi(q)$. Consider the level $Q$ remainder
\begin{equation}
    E(x,Q)=\sum_{q\leq Q}\max_{(a,q)=1}\max_{y\leq x}\abs{\psi(x,a,q)-\frac{x}{\varphi(q)}}.
\end{equation} Then $A(\theta)$ is the assumption that for any $A, \e>0$ 
\begin{equation}
    E(x,Q)=O_{\e, A}(x(\log x)^{-A}), \textrm{ when }Q=x^{\theta-\e}.
\end{equation}
Note that $A(1/2)$ is the Bombieri--Vinogradov theorem, and $A(1)$ is the Elliott--Halberstam conjecture. 

Friedlander and Iwaniec proved that there exists a $\theta_0<1$ such that, if   $A(\theta_0)$ is true, then 
\begin{equation} \label{hyperbolic-counting-prime-precise}
\pi_E(x)\asymp \frac{x}{\log x}.
\end{equation}
They furthermore conjectured that $\pi_E(x)\sim c\frac{x}{\log x}$ for some constant $c>0$. 

\subsection{Equidistribution over primes in the elliptic case}

To prove equidistribution over primes in this case we use the following result: 
\begin{thm}\label{crucialbounds} If $m=(m_1,m_2)\in \Z^2\backslash\{0\}$ and $l$ is a non-negative even integer, then 
\begin{equation}
\sum_{\substack{\g\in \G\\\nu_\H(\g)+l=p\leq x}}\exp({i(2\theta_1(\g)m_1+2\theta_2(\g)m_2)})=O_{m}\left(\frac{x}{\log x}\frac{1}{\log^{1-\frac{2}{\pi }}x}\right).
\end{equation}
\end{thm}

\begin{proof} Let
\begin{equation}
    A(m_1, m_2, x)=\sum_{\substack{\g\in \G\\\nu_\H(\g)+l=p\leq x}}\exp({i(2\theta_1(\g)m_1+2\theta_2(\g)m_2)}).
\end{equation} Then, by \eqref{def:S_e}, we have
\begin{equation}
    A(m_1,m_2,x)=\sum_{p\leq x-l}S_e(m_1,m_2,p+l).
\end{equation}
By using \eqref{eq:reduce to quotient} we may assume that $m_1$, $m_2$ are both even. We then use Lemma \ref{lem:factorisation} to deduce that
\begin{equation}\label{additive-expression}
    \abs{A(m_1,m_2,x)}\leq 4+8 \sum_{\substack{p\leq x \\p+l=3\bmod 4}}\abs{W^P_{m_1-m_2}(p+2)}\abs{W^P_{m_1+m_2}(p-2)}.
\end{equation}
Note that, since $W^P_{m}(n)=0$ for $n=3\bmod 4$, as follows from multiplicativity and \eqref{eq: W^P(p^l):p =3 mod 4}, we may as well sum over all odd primes $p\leq x$. 

Applying Theorem \ref{nair-tenenbaum} with the non-negative multiplicative functions \begin{equation}
    g_1(n)=\abs{W^P_{m_1-m_2}(n)}, \quad g_2(n)=\abs{W^P_{m_1+m_2}(n)},
\end{equation} and $(a_1,b_1)=(1,l+2)$, $(a_2,b_2)=(1,l-2))$, we arrive at 
\begin{equation}
    \abs{A(m_1,m_2,x)}\ll_{l}\frac{x}{(\log x)^3}\sum_{n\leq x}\frac{\abs{W_{m_1-m_2}^P(n)}}{n}\sum_{n\leq x}\frac{\abs{W_{m_1+m_2}^P(n)}}{n}.
\end{equation}
Using \eqref{a trivial-bound} the two sums are trivially bounded by $O(\log(x))$,
and since $m_1$, $m_2$ are not both zero at least one of  $m_1-m_2$, $m_1+m_2$ is non-zero. Call   this non-zero integer $m'$. It  is even, since both of $m_1$, $m_2$ are even. It now follows from Proposition \ref{Average-Euclidean Weyl sum bounded} that 
\begin{equation}
    \sum_{n\leq x}\frac{\abs{W^P_{m'}(n)}}{n}=
    O_{m'}((\log x)^{2/\pi}),
\end{equation}
which gives the result. 
\end{proof}

Since $1-2/\pi>0$ we can now conclude Theorems \ref{hyperbolic-equidistribution-prime} and\ref{hyperbolic-equidistribution-prime minus 2}   by specializing to $l=0$ and $l=2$,  using Weyl's equidistribution criterion in combination with \eqref{hyperbolic-counting-prime} and \eqref{hyperbolic-counting-prime-shifted}.

It follows from Theorem \ref{crucialbounds} that, if $m_1,m_2$ are not both zero, then, conditional on the lower bound in \eqref{hyperbolic-counting-prime-precise} we have 
\begin{equation}
    \frac{1}{\pi_E(x)}\displaystyle\sum_{\substack{\nu_\H(\g)=p\leq x}}\exp({i(2\theta_1(\g)m_1+2\theta_2(\g)m_2)})=O_{m}((\log x)^{2/\pi -1})
\end{equation}
which, via Weyl's equidistribution theorem, gives Theorem \ref{hyperbolic-equidistribution-prime}. 

\begin{rem}
Note that we do not really need the full force of the lower bound in \eqref{hyperbolic-counting-prime-precise}; we only need 
\begin{equation}
    \frac{1}{\pi_E(x)}
    =o\left(x^{-1}(\log x)^{2(1-\pi^{-1})}\right)
\end{equation}
to conclude equidistribution.
\end{rem}

\section{The hyperbolic case} \label{sec:hyperbolic}
We now shift our attention to a different subgroup of $G=\sl$; we consider the quaternion group 
\begin{equation}\label{quaternion group}
\Gamma(2,5)= \left\{\begin{pmatrix}x_0+x_1\sqrt{2} & \sqrt{5}(x_2+x_3\sqrt{2})\\
\sqrt{5}(x_2-x_3\sqrt{2})&x_0-x_1\sqrt{2}
\end{pmatrix}\in G\vert x_i\in \mathbb Z\right\}
\end{equation}
This is an embedding in $G$ of the standard order
$\mathcal O=\Z[1,i,j,k]$ in \begin{equation}\legendre{2,5}{\Q}=\{q_0+q_1i+q_2j+q_3j\vert q_i\in \Q \},\end{equation} where $i^2=2$, $j^2=5$, $ij=-ji=k$.

It is well known that $\Gamma(2,5)$ is a discrete strictly hyperbolic co-compact subgroup of $G$ (See \cite[p 302-303]{Hejhal:1976a}). It has genus $3$ and co-volume $8\pi$. 
It contains the primitive hyperbolic subgroup 
\begin{equation} 
H=\langle h_0\rangle \textrm{ generate by }h_0=\begin{pmatrix}
    \varepsilon^2 &\\ &\varepsilon^{-2},
\end{pmatrix}
\end{equation} where $\varepsilon=1+\sqrt 2$ is the fundamental unit the ring of integers of the number field $\mathbb{Q}(\sqrt{2})$. 

\subsection{Hyperbolic decomposition in \texorpdfstring{$\sl$}{SL2(R)} and Good's equidistribution}
In this section we make a simplified exposition of the decomposition given in \cite[Lemma 1]{Good:1983b} leading up to Good's theorem in the case of both subgroups being hyperbolic.

Let $g=\begin{pmatrix}
    a&b\\c&d
\end{pmatrix}\in G=\sl$ and consider
\begin{align}
s&=\{g\in G\vert\ gz_1=z_2\textrm{ for some }z_1,z_2\in \{0,i \infty\} \}\\
S&=s\cup \{g\in G\vert\ g(iy_1)=iy_2\textrm{ for some }y_1,y_2\in \R_+  \}.
\end{align}
It is straightforward to check that $s=\{g\in G\vert abcd=0 \}.$

\begin{lem} \label{decomposition-step}Let $g\in G\backslash s$. Then there exist unique $y_1,y_2> 0$ such that 
\begin{equation}
g=\begin{pmatrix}y_1^{1/2}&\\ &y_1^{-1/2}\end{pmatrix}\begin{pmatrix}\alpha&\beta\\\gamma&\delta\end{pmatrix}\begin{pmatrix}y_2^{1/2}&\\ &y_2^{-1/2}\end{pmatrix}
\end{equation} with the middle matrix satisfying $\abs{\alpha}=\abs{\delta}$, $\abs{\beta}=\abs{\gamma}$.
\end{lem}
\begin{proof} Since $abcd\neq 0$, the matrix
\begin{equation}
\begin{pmatrix}y_1^{-1/2}&\\ &y_1^{1/2}\end{pmatrix}\begin{pmatrix}a&b \\c & d\end{pmatrix}\begin{pmatrix}y_2^{-1/2}&\\ &y_2^{1/2}\end{pmatrix}
=\begin{pmatrix}(y_1y_2)^{-1/2}a&(y_2/y_1)^{1/2} b \\ (y_1/y_2)^{1/2} c & (y_2 y_1)^{1/2}d\end{pmatrix}
\end{equation}
has the desired form if and only if $y_1y_2=\abs{a/d}$ and $y_1/y_2=\abs{b/c}$. This has a unique strictly positive solution given by
\begin{equation}
    y_1=\abs{\frac{ab}{cd}}^{1/2}, \quad \quad y_2=\abs{\frac{ac}{bd}}^{1/2}.
\end{equation}
\end{proof}
Recall the matrix
\begin{equation}
\omega=\begin{pmatrix}
0 & -1 \\ 1 & 0 \end{pmatrix}.
\end{equation}
\begin{lem}\label{decomposition-step-two}Let $g\in G\backslash s$ and assume that $\abs{a}=\abs{d}, \abs{b}=\abs{c}$. Then either $g\in K$  or there exist uniquely determined numbers $\delta_1, \delta_2\in \{0,1\}$, $v> 0$, and a sign $\pm$ such that 
\begin{equation}g=\pm \omega^{\delta_1}\begin{pmatrix}
    \cosh v&\sinh v\\ \sinh v& \cosh v
\end{pmatrix} \omega^{\delta_2} .\end{equation}
Moreover, the matrix $g\in K$ if and only if $g\in S\backslash s$, which happens if and only if $abcd<0$.
\end{lem}
\begin{proof} By the given assumptions we are in exactly one of the following four cases:
\begin{enumerate}
    \item \label{case-1} $a=d$ and $b=-c$, such that $c^2+d^2=1$.
    \item \label{case-2} $a=d$ and $b=c$, such that $d^2-c^2=1$.
    \item \label{case-3} $a=-d$ and $b=-c$, such that $-d^2+c^2=1$.
    \item \label{case-4} $a=-d$ and $b=c$, such that $-d^2-c^2=1$.
\end{enumerate}

In case \eqref{case-1} $g\in K$. 

In case \eqref{case-2} and \eqref{case-3} we compute $\omega^{-\delta_1}g\omega^{-\delta_2}$ for all four values of $\delta_1, \delta_2$. In case \eqref{case-2} these are \begin{align}
g&=\begin{pmatrix}d & c\\c& d\end{pmatrix}, 
& 
\omega^{-1} g&=\begin{pmatrix}c & d\\-d& -c\end{pmatrix},\\
g\omega^{-1} &=\begin{pmatrix}-c & d\\-d& c\end{pmatrix}, 
& 
\omega^{-1} g\omega^{-1}&=\begin{pmatrix}-d & c\\c& -d\end{pmatrix}.\\ 
\end{align}
Precisely one choice of $\delta_1, \delta_2$  has all entries to be of the same sign, namely $(\delta_1, \delta_2)=(0,0)$, if $c$ and $d$ already has the same sign, and $(\delta_1, \delta_2)=(1,1)$ if they have opposite sign. With this choice we have 
\begin{equation}
    g=\pm \omega^{\delta_1}\begin{pmatrix}
        \abs{d}&\abs{c} \\
        \abs{c}&\abs{d} \\
    \end{pmatrix}\omega^{\delta_2}\end{equation} for a unique choice of $\pm$. Since $d^2-c^2=1$, we have that $\abs{d}>1$, and we have the claimed decomposition with $v=\log(\abs{d}+\abs{c})>0$.

Similarly in case \eqref{case-3}
\begin{align}
g&=\begin{pmatrix}-d & -c\\c& d\end{pmatrix}, &
\omega^{-1} g&=\begin{pmatrix}c & d\\d& c\end{pmatrix}\\
g\omega^{-1} &=\begin{pmatrix}c & -d\\-d& c\end{pmatrix}, &
\omega^{-1} g\omega^{-1}&=\begin{pmatrix}-d & c\\-c& d\end{pmatrix}.\\ 
\end{align}
Again, precisely one choice of $\delta_1, \delta_2$  has all entries to be of the same sign, namely $(\delta_1, \delta_2)=(1,0)$, if $c$ and $d$ already have the same sign, and $(\delta_1, \delta_2)=(0,1)$, if they have opposite signs. With this choice we have 
\begin{equation}
    g=\pm \omega^{\delta_1}\begin{pmatrix}
        \abs{c}&\abs{d} \\
        \abs{d}&\abs{c} \\
    \end{pmatrix}\omega^{\delta_2}.\end{equation} Since $c^2-d^2=1$, we have that $\abs{c}>1$ getting the claimed decomposition again         with $v=\log(\abs{d}+\abs{c})>0$.

Case \eqref{case-4} does not happen since $-c^2-d^2=1$ does not have any real solutions. 

To see the final claim note that if $g\in K$ then $gi=i$ so $g\in S\backslash s$. If, on the other hand, $g\in S\backslash s$, then there exist $y_1,y_2>0$ such that $g(iy_1)=iy_2$, which implies $aiy_1+b=-cy_1y_2+diy_2$. It follows that $ay_1-dy_2=0=-b-cy_1y_2$,
which is only possible if $a,d$ has the same sign, and $b,c$ has opposite signs. Hence we are in case \eqref{case-1} and $g\in K$. Finally we finish the proof by noting that $abcd<0$ precisely in case \eqref{case-1}.
\end{proof}

Noticing that $s$ and $S\backslash s$ are closed under multiplication from the left and the right by multiplication by $\begin{pmatrix}y^{1/2} &\\&y^{-1/2}\end{pmatrix}$, we may use the previous lemmata and the formulas in their proofs to conclude the following lemma.
\begin{lem}\label{full-hyperbolic-decomposition} For $g\in G\backslash s$ we have 
$g\in S$ if and only if $\abs{ad}+\abs{bc}=1$. If $g\notin S$ then there exist  uniquely determined numbers $y_1, y_2>0$, $\delta_1,\delta_2\in \{0,1\}$, $v>0$, and sign $\pm$ such that 
\begin{equation}g=\pm \begin{pmatrix}y_1^{1/2}&\\ &y_1^{-1/2}\end{pmatrix} \omega^{\delta_1}\begin{pmatrix}
   \cosh v&\sinh v\\ \sinh v& \cosh v
\end{pmatrix} \omega^{\delta_2} \begin{pmatrix}y_2^{1/2}&\\ &y_2^{-1/2}\end{pmatrix}.\end{equation}
Concretely 
\begin{align}
  y_1&=\abs{\frac{ab}{cd}}^{1/2},\quad y_2=\abs{\frac{ac}{bd}}^{1/2}, \\v&=\log(\abs{ad}^{1/2}+\abs{cb}^{1/2}),\\
  (\delta_1,\delta_2)&=\begin{cases}
  (0,0),& \textrm{ if } \hbox{sign}(a,b,c,d)=\pm(+,+,+,+),\\
  (1,1),& \textrm{ if } \hbox{sign}(a,b,c,d)=\pm(+,-,-,+),\\
  (1,0),& \textrm{ if } \hbox{sign}(a,b,c,d)=\pm(+,+,-,-),\\
  (0,1),& \textrm{ if } \hbox{sign}(a,b,c,d)=\pm(+,-,+,-).\\
  \end{cases}
\end{align}
\end{lem}
\begin{proof}
To see the condition for $g$ to be in $S$ we apply Lemma \ref{decomposition-step} to $g$. Let 
\begin{equation}
g'=\begin{pmatrix}y_1^{-1/2}&\\ &y_1^{1/2}\end{pmatrix}\begin{pmatrix}a&b \\c & d\end{pmatrix}\begin{pmatrix}y_2^{-1/2}&\\ &y_2^{1/2}\end{pmatrix}
=\begin{pmatrix}\frac{a\abs{d}^{1/2}}{\abs{a}^{1/2}}&\frac{b\abs{c}^{1/2}}{\abs{b}^{1/2}} \\ \frac{c\abs{b}^{1/2}}{\abs{c}^{1/2}} & \frac{d\abs{a}^{1/2}}{\abs{d}^{1/2}}\end{pmatrix}.
\end{equation} 
Then  
 $g\in S$ if and only if $g'\in S\backslash s$. 
Since $g'$ satisfies the assumptions of Lemma \ref{decomposition-step-two},  $g'\in S\backslash s$ if and only if $g'\in K$. This is equivalent to $a,d$ having the same sign and $b,c$ having opposite signs. But then the determinant condition gives $\abs{ad}+\abs{bc}=ad-bc=1$. In the opposite direction, if we assume that $\abs{ad}+\abs{bc}=1$ then combining this with the determinant condition we find that $\abs{ad}-ad=-(bd+\abs{bd})$. But since the left-hand side is non-negative and the right-hand side is non-positive both sides are zero. This implies that $a,d$ has the same sign, and $b,d$ has opposite signs. Hence $g'\in K$ and we conclude that $g\in S$. Alternatively, we may use Lemma \ref{decomposition-step-two} to conclude that $g\in S$ if and only if $ad$ and $cd$ has opposite signs, which by the determinant condition happens precisely  if $ad>0$ so that $1=ad-bc=\abs{ad}+\abs{bc}$.

To see the decomposition of $g\notin S$ we apply Lemma \ref{decomposition-step-two} to $g'$.
\end{proof}

For the decomposition of a matrix $g\in G$ we will often write the parameters $v=v(g)$, $y_i=y_i(g)$, and $\delta_i=\delta_i(g)$. 

Martin, McKee, and Wambach \cite{MartinMcKeeWambach:2011} introduced a different parameter \begin{equation}\label{MMW-parameter}
    \delta(g)=2\abs{ad+bc}. 
\end{equation} We now describe how this relates to $v(g)$ when $g\in G\backslash S.$ Since $\g\notin S$ we have $abcd>0$, so $ad$ and $bc$ has the same sign. We deduce  that 
\begin{align}
1<e^{2v(g)}&=(\abs{ad}^{1/2}+\abs{bc}^{1/2})^2\\
& =\abs{ad}+\abs{bc}+\sqrt{4abcd}\\
& =\abs{ad+bd}+\sqrt{\abs{ad+bc}^2-1}\\
& =\delta(g)/2+\sqrt{(\delta(g)/2)^2-1}=e^{\arccosh{\delta(g)/2}},
\end{align}
so that 
\begin{equation}\label{MMW-parameter-relation}
    v(g)=\frac{1}{2}\arccosh\left(\frac{\delta(g)}{2}\right).
\end{equation}
We use this to give the following geometric interpretation of the parameter $v(g)$ as the closest hyperbolic distance between the vertical geodesic from 0 to $i\infty$ and its image under $g$.

\begin{prop} Let $g\in G\backslash S$. Then  
\begin{equation}
    v(g)=\frac{1}{2}d_\H(gi\R_+,i\R_+)>0.
\end{equation}
\end{prop}
\begin{proof}
This follows from \cite[Lemma 1]{MartinMcKeeWambach:2011} and the above identification. 
\end{proof}

Martin, McKee, and Wambach \cite[Proof of Lemma 1]{MartinMcKeeWambach:2011} found that if $g \in G\backslash S$ is decomposed as in Lemma \ref{full-hyperbolic-decomposition}, then the distance $d_\H(gi\R_+,i\R_+)$ is attained between $iy_1$ on the vertical geodesic and $giy_2^{-1}$ on the second, i.e. \begin{equation}d_\H(i\R_+,\gamma (i\R_+)=d_\H(iy_1,g iy_2^{-1}).\end{equation}
We can now describe Good's theorem. Let $\G$ be a discrete co-finite subgroup of $\sl$, and let $\g_1, \g_2 \in \G$ be primitive hyperbolic elements. Fix scaling elements $\sigma_{1}, \sigma_2\in G$ satisfying \begin{equation}\g_l=\sigma_l^{-1} \begin{pmatrix} m_l&\\&m_l^{-1}\end{pmatrix}\sigma_{l}\end{equation} with $1<\abs{m_l}<\infty$, and write $\Gamma_{m_l}=\sigma_l\langle\g_l\rangle\sigma_l^{-1}.$ Let \begin{equation}0=\lambda_0< \lambda_1\leq \cdots \leq \lambda_N<1/4\end{equation} be the eigenvalues of the automorphic Laplacian on $\GmodH$ below 1/4, and write $\lambda_j=s_j(1-s_j)$ with $s_j>1/2$. We now quote Good \cite[Thm 4, p.116]{Good:1983b} in this special case: 
\begin{thm} \label{goods-theorem} Let $\delta_1, \delta_2\in \{0,1\}$, and $n=(n_1,n_2)\in \Z^2$.  There exist explicit complex constants $c_{1}, \ldots, c_N$ such that          
\begin{align}
\sum_{\substack{
\delta_l(\g)=\delta_l\\
e^{v(\g)}\leq X^{1/2} }}
\!\!\!\!\!\!
e\left(n_1\frac{\log y_1(\g)}{\log m_1^2}+ n_2\frac{\log y_2(\g)}{\log m_2^2}\right)=\delta_{n,0}&\frac{\log m_1^2\log m_2^2}{4\pi\vol{\GmodH}}X\\ &+ \sum_{j=1}^Nc_j X^{s_j}+O_{n}(X^{2/3}).
\end{align}
Here the sum is over double cosets of $\Gamma_{m_1}\backslash (\sigma_1\Gamma\sigma_2^{-1} \cap S^c)\slash\Gamma_{m_2}      $. 
\end{thm}
By Weyl's equidistribution theorem it follows that \begin{equation}\left(\frac{\log y_1(\g)}{\log m_1^2}, \frac{\log y_2(\g)}{\log m_2^2}\right)\end{equation} is equidistributed on $(\R\slash\Z)^2$.

\subsection{The ring of integers of \texorpdfstring{$\mathbb{Q}(\sqrt{2})$}{Q(sqrt(2))}}
The arithmetic properties of $\Gamma(2,5)$ are to a large extent controlled by the real quadratic field $K=\mathbb Q(\sqrt{2})$, and its ring of integers $\mathcal O_K$.

The ring is a principal ideal domain, has regulator $\log(\varepsilon)$, discriminant $8$, class number one, and has $\pm 1$ as its only roots of unity. Moreover, any unit $u\in \mathcal O_K$ is of the form $u=\pm \varepsilon^n$ for some $n\in \mathbb Z$. The ring of integers $\mathcal O_K$ comes equipped with an element norm \begin{equation}N(z)=z\cdot \sigma z\in \Z, \textrm{ where } \sigma(x+\sqrt{2}y)=x-\sqrt{2}y,\end{equation} 
and an ideal norm $N((z))=\abs{N(z)}$. The Galois group is generated by the involution $\sigma$. 

Recall that $\mathcal O_K$ is Euclidean and, therefore, a unique factorisation domain; its irreducible elements are 
\begin{enumerate}[label=\roman*)]
\item $\sqrt{2}$,
\item $\rho=a+\sqrt{2}b$ with  $N(\rho)=p=\pm 1\bmod 8$,  
\item $p = \pm 3\pmod 8$,
\end{enumerate}
and their associates. 

The Dedekind zeta function of $K$ factors as 
\begin{equation}
\zeta_K(s)=\zeta(s)L(s, \chi_8),
\end{equation} where $\chi_8(n)=\legendre{8}{p}$ is the Kronecker symbol. The character $\chi_8(n)$ is the unique primitive even Dirichlet character modulo 8. By the class number formula \begin{equation}\label{L-at-1}L(1, \chi_8)=\frac{\log \e}{\sqrt{2}}.\end{equation}

\begin{prop}\label{totally-positive-generator}
    Every  ideal $I\subset \mathcal O_K$ has a totally positive generator, i.e.
    \begin{equation}
        I=(z), \quad\textrm{for some } z\in \mathcal O_K\textrm{ with } z,\sigma(z)>0. 
    \end{equation}
Given two such generators $z_1,z_2$, there exists $m\in \mathbb Z$ such that $z_1=\e^{2m}z_2.$
\end{prop}
\begin{proof}
Since $\mathcal O_K$ is a principal ideal domain, the ideal $I$  has a generator $w$. For any unit $u$ the element $uw$ is another generator and if $w, \sigma(w)$ has different signs then $\varepsilon w, \sigma(\varepsilon w)$ has the same sign since $\sigma(\varepsilon)$ is negative. So, without loss of generality, we may assume that $w, \sigma(w)$ has the same sign. If this sign is positive we are done, and if not $-w$ has the desired property.

If two $z_1, z_2$ are totally positive generators for $I$, then $z_1=uz_2$ for some unit $u=\pm \varepsilon^n$.  Since $z_1$ and $z_2$ are totally positive, this is only possible if $u=\varepsilon^{2m}$. 
\end{proof}

Consider now the set of classes of totally positive elements of $\mathcal O_K$ with norm $n$ modulo the equivalence relation $z_1\sim z_2$ if and only if $z_1=\epsilon^{2m}z_2$ for some $m\in \Z$:
\begin{equation}
    D_K(n)=\{z\in \mathcal O_K\vert N(z)=n, z>0, \sigma(z)>0\}\slash \sim.
\end{equation}

For $n>0$ this is in bijection with the set of ideals of $\mathcal O_K$ with norm $n$ via the map $z\mapsto (z)$, as follows from Proposition \ref{totally-positive-generator}. In particular 
\begin{equation}
\label{counting ideals}   \mathcal N_2(n):= \# D_K(n)=\#\{I\subseteq \mathcal O_K\vert N(I)=n\}= \sum_{d\vert n}\chi_{8}(n).
\end{equation} The last equality follows  from the divisibility theory of $K$; see \cite[Satz 882]{Landau:1969a}.

\subsection{Analysis of Weyl sums in  \texorpdfstring{$\mathcal{O}_K$}{OK}}\label{analysis-weyl-sums-hyperbolic}
Consider, for $k$ even,  
\begin{equation}U_k(n)=\sum_{z\in D_K(n)}\lambda(z)^k,
\end{equation}
where \begin{equation}\label{sqrt-basic-gross}\lambda(z)=\abs{\frac{z}{\sigma z}}^{\frac{\pi i}{4\log(\varepsilon)}}\end{equation} is the square root of the basic Gr\"ossencharacter in $\mathcal O_K$, see \cite[§ 10]{Hecke:1920}.
It is straightforward to verify that this is well-defined and that for $n>0$
\begin{equation}\label{another-trivial-bound}
\abs{U_k(n)}\leq U_0(n)=\# \{I\subseteq \mathcal O_K: N(I)=n\}=\sum_{d\vert n} \chi_8(d),
\end{equation}
as follows from \eqref{counting ideals}. In particular $\abs{U_k(p^l)}\leq l+1$, and $\sum_{n\leq x}\abs{U_k(n)}= O(x)$.
If $n_1,n_2\in \N$ are coprime, then the map
\begin{equation}
\begin{array}{ccc}
D_K(n_1)\times D_K(n_2)&\to & D_K(n_1n_2)\\
(z_1,z_2)&\mapsto& z_1z_2
\end{array}
\end{equation}
is an isomorphism. This implies that $U_k(n)$ is multiplicative as a function of $n$. Using the factorisation into irreducible elements we see that 
\begin{equation}
U_k(p)=\begin{cases}
1,&\textrm{ if }p=2,\\
2\cos{\left(\frac{\pi k}{4\log \e}\log\abs{\frac{\rho_p}{\sigma(\rho_p)}}\right)},&\textrm{ if }p=\pm 1\bmod 8,\\
0,&\textrm{ if }p=\pm 3 \bmod 8
.
\end{cases}
\end{equation}
Here the element $\rho_p$ is \emph{any} of the two  totally positive elements modulo $\sim$ satisfying $N(\rho_p)=p=\pm 1\bmod 8$, i.e. $\rho_p$ is any of the two elements of $D_K(p)$. The Galois element $\sigma$ permutes these elements. Since cosine is an even function, the expression above is independent of  the  element we choose. 

In anticipation of  using Theorem \ref{weak-halberstam-richert} we  find the average size of $\abs{U_k(p)}$. This can be done using an effective version of Hecke's equidistribution theorem \cite[Sec. 7]{Hecke:1920}. The following is an adaptation  of a classical result due to Rademacher. Radamacher proved it using good zero-free regions for Hecke $L$-series; a more precise error term was found by Urbjalis \cite{Urbjalis:1964}.
\begin{thm}[Rademacher \cite{Rademacher:1935}]\label{Rademacher} There exists a constant $b>0$ such that for an interval $I\subseteq \R\slash\Z$ we have
\begin{equation}\#\left\{\rho \in D_K(p)\left\vert \begin{matrix}p=\pm 1\bmod 8\leq x \\  \frac{\log\abs{\frac{\rho}{\sigma(\rho)}}}{2\log\e^2 }\in I \end{matrix}\right.\right\}=\abs{I}\li(x)+O(xe^{-b\sqrt{\log x}}).\end{equation}
The implied constant depends only on $K$.
\end{thm}
Using this effective equidistribution theorem we can prove the following result:
\begin{lem}\label{corollary-of-Hecke-equidistribution}
   For a non-zero $k\in2\Z$ with  $\log\abs{k}\leq b\sqrt{\log x}$   we have 
   \begin{equation}
      \sum_{\substack{p\leq x\\ p=\pm 1\bmod 8}}\frac{\abs{\cos \left(\frac{\pi k}{4\log \e}\log\abs{\frac{\rho_p}{\sigma(\rho_p)}}\right)}}{p}\leq \frac{1}{\pi}\log\log x+\left(1-\frac{2}{\pi}\right)\log\log k+O(1).
   \end{equation}
\end{lem}
\begin{proof} We follow the strategy in \cite[p. 91-92]{ErdosHall:1999}. By possibly shifting to $-k$ we may assume $k>0$. We start by showing that for all even $k$ we have 
\begin{equation} \label{intermediate-expression}
      \frac{1}{2}\sum_{\substack{p\leq x\\ p=\pm 1\bmod 8\\\rho\in D_K(p)}}{\abs{\cos \left(v_k(\rho)\right)}}=\frac{1}{\pi}\li(x)+O(kxe ^{-b\ \sqrt{\log x}}),
   \end{equation}
where $v_k(\rho)=\frac{\pi k}{4\log \e}\log\abs{\frac{\rho}{\sigma(\rho)}}$. Note that we are summing over both $\rho_p$ and $\sigma(\rho_p)$. This makes it easier to use Rademacher's theorem. 
We choose representatives for $\rho$ such that $v_k(\rho)\in [-\pi/2,\pi k-\pi/2[=E$. We split this interval as a disjoint union according to the sign of $\cos(v)$ i.e.
\begin{equation}
 E=\bigcup_{n=0}^{k-1}E_n\textrm{ where }E_n=[-\frac{\pi}{2}+n\pi ,\frac{\pi}{2}+n\pi[,
\end{equation} and note that $\abs{\cos(v)}=(-1)^n\cos(v)=(-1)^n\int_v^{\frac{\pi}{2}+n\pi}\sin(\theta)d\theta$ for $v\in E_n$. It follows that 
\begin{align}
&\sum_{\substack{p\leq x\\ p=\pm 1\bmod 8\\\rho\in D_K(p)}}\abs{\cos \left(v_k(\rho)\right)}=\sum_{n=0}^{k-1}(-1)^n\sum_{\substack{p\leq x\\ p=\pm 1\bmod 8}}\sum_{\substack{\rho\in D_K(p)\\v_k(\rho)\in E_n}}\int_{v_k(\rho)}^{\frac{\pi}{2}+n\pi}\sin(\theta)d\theta\\
&\quad  =\sum_{n=0}^{k-1}(-1)^n \int_{E_n}(\sum_{\substack{p\leq x\\ p=\pm 1\bmod 8}}\sum_{\substack{\rho\in D_K(p)\\ -\frac{\pi}{2}+n\pi\leq v_k(\rho)\leq \theta }}1)\sin(\theta)d\theta.\\
\intertext{Using Theorem \ref{Rademacher} on the inner sum we find}
&\quad =\sum_{n=0}^{k-1}(-1)^n \int_{E_n}\left(\left(\frac{\theta}{k\pi}-\frac{n-1/2}{k}\right)\li(x)+O(xe^{-c\sqrt{\log x}})\right)\sin(\theta)d\theta.
   \end{align}
Using $\int_{E_n}\sin(\theta)d\theta=0,$ and $\int_{E_n}\theta \sin(\theta)d\theta=2(-1)^n$, we arrive at \eqref{intermediate-expression}. 
It follows that for any $2<w\leq x$ we have 
\begin{align}\label{another-intermediate}
\sum_{\substack{p\leq x\\ p=\pm 1\bmod 8}}\frac{1}{p}{\abs{\cos \left(\frac{\pi k}{4\log \e}\log\abs{\frac{\rho_p}{\sigma(\rho_p)}}\right)}}\leq &\frac{1}{2}\log\log w+O(1)\\&+\frac{1}{\pi}\log\left(\frac{\log x}{\log w}\right) +O(ke^{-b\sqrt{\log w}}).
\end{align}
Here we have estimated the sum up $w$ trivially, and used partial summation and \eqref{intermediate-expression} on the rest. After that we used $\int \li(x)/x^2dx=\log\log x+O(1).$ If we choose $w$ subject to $\log(w)=(b^{-1}\log k)^2$, then the last term is bounded. The condition $w\leq x$ gives the condition on $k$, and we arrive at the claim.

\end{proof}
\begin{lem}\label{powerful-bound}
Let $k$ be an even non-zero integer. Then there exists a constant $b$ such that for $\log\abs{k}\leq b\sqrt{\log x}$ we have
\begin{equation}
    \sum_{n\leq x}\abs{U_k(n)}=O\left(x\left(\frac{\log^2\abs{k}}{\log x}\right)^{1-2/\pi}\right).
\end{equation}
\end{lem}
\begin{proof}We use Theorem \ref{weak-halberstam-richert} with $f(n)=\abs{U_k(n)}$. The relevant assumptions were checked above. Using \begin{equation}
\sum_{p\leq x}\frac{1}{p}=\log\log x+O(1),
\end{equation}
and Lemma \ref{corollary-of-Hecke-equidistribution} the claim follows. 
\end{proof}

\subsection{Parametrisation  of \texorpdfstring{$\G(2,5)$}{G(2,5} }In this section we return 
 to the quaternion group $\G(2,5)$, see \eqref{quaternion group}, and show how we can parametrise parts of it using $\mathcal O_{K}$, where $K=\Q(\sqrt{2})$. Consider for $\g\in \Gamma(2,5)$ the two algebraic integers 
\begin{align}
z_1(\g)&=x_0+x_1\sqrt{2}\in \mathcal O_K, \\
z_2(\g)&=x_2+x_3\sqrt{2} \in \mathcal O_K.
\end{align}
By the determinant condition we deduce that $N(z_1(\g))-5N(z_2(\g))=1$ so $N(z_1(\g))=5N(z_2(\g))+1$. It is now straightforward to verify the following proposition:
\begin{prop}\label{prop:basic-hyperbolic}\phantom{123}
\begin{enumerate}
    \item \label{prop:basic-hyperbolic-injective}The map $\g\mapsto (z_1(\g),z_2(\g)) $ is injective. 
    \item If $\g\in \G\cap S^c$ then $N(z_1(\g))$, $N(z_2(\g))$ are non-zero and have the same sign. Writing $N(z_1(\g))=5N(z_2(\g))+1=5n+1$, we have the following statements.   \begin{enumerate}
    \item \label{prop:basic-hyperbolic-length} if $\delta_1(\g)+\delta_2(\g)$ is even, then $n\in \N$ and $\delta(\g)=20n+2$.
    \item if $\delta_1(\g)+\delta_2(\g)$ is odd, then $n\in -\N$ and $\delta(\g)=-20n-2$.
    \end{enumerate}
    \item \label{prop:basic-hyperbolic-quotients}If $\g'=h_0^{j_1}\g h_0^{j}$, then 
    \begin{align}
            z_1(\g')&=\varepsilon^{2(j_1+j_2)}z_1(\g),\\
            z_2(\g')&=\varepsilon^{2(j_1-j_2)}z_2(\g).
    \end{align}
   
    \item \label{prop:basic--hyperbolic-angles}We have
     \begin{align}
e\left(\frac{\log y_1(\g)}{\log \varepsilon^4}\right)
    &=\left(\lambda(z_1(\g))\lambda(z_2(\g))\right),\\ 
e\left(\frac{\log y_2(\g)}{\log \varepsilon^4}\right)
    &=\left(\lambda(z_1(\g))/\lambda(z_2(\g))\right),\\ 
    \end{align}
where $\lambda(z)$ is the square root of the basic Gr\"ossencharacter in $\mathcal O_K$, see \eqref{sqrt-basic-gross}, \cite[§ 10]{Hecke:1920}.
    
\end{enumerate}
\end{prop}

Note that Proposition \ref{prop:basic-hyperbolic} \eqref{prop:basic--hyperbolic-angles} implies that
\begin{align}
\frac{\log y_1(\g)}{\log\e^4}&=\frac{1}{2}\left(\frac{\log\abs{\frac{z_1(\g)}{\sigma z_1(\g)}}}{\log \e^4}+\frac{\log\abs{\frac{z_2(\g)}{\sigma z_2(\g)}}}{\log \e^4}\right)\bmod 1,\\
\frac{\log y_2(\g)}{\log\e^4}&=\frac{1}{2}\left(\frac{\log\abs{\frac{z_1(\g)}{\sigma z_1(\g)}}}{\log \e^4}-\frac{\log\abs{\frac{z_2(\g)}{\sigma z_2(\g)}}}{\log \e^4}\right)\bmod 1.
\end{align}

We want to describe the intersection of $\G(2,5)$ with $S\backslash s$ and $s$. Note that if $\g=\begin{pmatrix}
    a&b\\c&d
\end{pmatrix}$ then $abcd=N(z_1)5N(z_2)=(5N(z_2)+1)5N(z_2).$ If $N(z_2)\leq 0$ then $(5N(z_2)+1)<0$ so we have $abcd\geq 0$. If, on the other hand,  $N(z_2)\geq 0$ then $(5N(z_2)+1)>0$ so also in this case $abcd\geq 0$. 

We see that $abcd$ vanishes if and only in $N(z_2)=0$. This happens precisely if $N(z_1)=1$ i.e. if $z_1$ is a unit with norm 1. But this means that $\g$ is a power of $h_0$. 

Summarizing we have shown that 
\begin{equation}
 \G(2,5)\cap s= H \textrm{ and } \G(2,5)\cap (S\backslash s)= \emptyset.  
\end{equation}

For a double coset $[\g]\in H\backslash (\G(2,5)-H)\slash H$ we note that we have $\delta_1(\g)=\delta_2(\g)=0$ precisely if all the four entries of $\g$ has the same sign. 
\begin{thm}\label{a cool map} The map
\begin{equation}
    \begin{array}{ccc}
    \psi:\{[\g]\vert {\delta_1(\g)=\delta_2(\g)=0, bc=5n}\}\slash\{\pm I\}&\to& D_K(5n+1)\times D_K(n)\\
    \g&\mapsto & (\abs{z_1(\g)}, \abs{z_2(\g)})
    \end{array}
\end{equation} is well-defined, two-to-one, and surjective. 
\end{thm}
\begin{proof}
It follows from Proposition \ref{prop:basic-hyperbolic} that $\psi$ is well-defined. 

Given $(z_1,z_2)\in D_K(5n+1)\times D_K(n)$ we consider the matrix 
\begin{equation}
    \gamma=\begin{pmatrix}
    z_1 &\sqrt{5}z_2\\
    \sqrt{5}\sigma(z_2)&\sigma(z_1)
\end{pmatrix}\in \G,
\end{equation} which satisfies that $\delta_1(\g)=\delta_2(\g)=0$ and $\sqrt{5}z_2\sqrt{5}\sigma(z_2)=5n$. This shows $\psi$ is surjective.

To see that $\psi$ is two-to-one we note that the two matrices in $\G(2,5)$ given by \begin{equation}\label{two-matrices}
\g=\begin{pmatrix}
    z_1 &\sqrt{5}z_2\\
    \sqrt{5}\sigma(z_2)&\sigma(z_1)
\end{pmatrix}, \quad \g'=\begin{pmatrix}
    z_1\varepsilon^2 &\sqrt{5}z_2\\
    \sqrt{5}\sigma(z_2)&\sigma(z_1\varepsilon^2) 
\end{pmatrix}\end{equation}
represent different double cosets and map to the same element in $D_K(5n+1)\times D_K(n)$. Assume now that $\psi (\g_1)=\psi(\g_2)$. By possibly taking minus the matrix we may assume that all entries of $\g_1$ and $\g_2$ are positive. It follows that there exist integers $n_1, n_2$ such that
\begin{align}
z_1(\g_2)&=z_1(\g_1)\varepsilon ^{2n_1},\\
z_2(\g_2)&=z_2(\g_1)\varepsilon ^{2n_2}. 
\end{align}
Write $z_i(\g_1)=z_i$.  Using that 
\begin{equation}h_0^{j_1}\gamma_1 h_0^{j_2}=\gamma=\begin{pmatrix}
    z_1\varepsilon^{2(j_1+j_2)} &\sqrt{5}z_2\varepsilon^{2(j_1-j_2)} \\
    \sqrt{5}\sigma(z_2\varepsilon^{2(j_1-j_2)})&\sigma(z_1\varepsilon^{2(j_1+j_2)})\end{pmatrix} ,\end{equation}
we see that, if $n_1$, $n_2$ have the same parity, then $\g_1$ and $\g_2$ are in the same double coset, and, if $n_1, n_2$ have different parity, then $\g_2$ is in the same double coset as the second matrix in \eqref{two-matrices}. This shows that the map is two-to-one.
\end{proof}

\subsection{Counting double cosets with prime norm in the hyperbolic case}
In order to find asymptotics for the number of double cosets with prime norm in the hyperbolic case we first prove a variant of the Titchmarsh divisor problem. Recall that in the Titchmarsh divisor problem \cite{Titchmarsh:1930} we want to determine asymptotics for sums like
 \begin{equation} \sum_{n\leq x}1*1(n+1)\Lambda(n),\end{equation}
where $*$ denotes the usual Dirichlet convolution between arithmetical functions.

 Consider 
\begin{equation}\psi(x;q,a)=\displaystyle\sum_{\displaystyle
\substack{n\leq x\\ n=a\bmod q}}\Lambda(n).\end{equation} The famous Bombieri--Vinogradov theorem \cite{Bombieri:1965}, see also \cite{Vaughan:1980}, states that for every $A>0$ there exist a $B>0$ such that 
\begin{equation}
\label{Bombieri-Vinogradov}\sum_{q\leq Q}\max_{\substack{(a,q)=1\\ y\leq x}}\abs{\psi(y,q,a)-\frac{y}{\varphi(q)}}=O_A(\frac{x}{\log^A(x)}),
\end{equation} for $Q=O(\frac{x^{1/2}}{\log^B x})$. This suffices - when combined with the Brun--Titchmarsh inequality 
\begin{equation}\label{Brun-Titchmarsh}
\pi(x+y,q,a)-\pi(x,q,a)<\frac{2y}{\varphi(q)\log(y/q)}, \textrm{ for }(a,q)=1, q< y,
\end{equation}
see \cite[Thm 6.6]{IwaniecKowalski:2004a}, to find the main term and an error term of the form $\frac{x\log \log x}{\log x})$ See \cite{Rodriquez:1965} \cite[Thm 3.9]{HalberstamRichert:1974a}.  
 In order to get better error term estimates, one need to extend the validity of bounds like \eqref{Bombieri-Vinogradov}. In the case of the classical Titchmarsh divisor problem this was done independently by Fouvry \cite{Fouvry:1985} and by Bombieri, Friedlander and Iwaniec \cite{BombieriFriedlanderIwaniec:1986}. They found that for any $A>0$
 \begin{equation} \sum_{n\leq x}1*1(n+1)\Lambda(n)=c_1 x\log x +c_2x +O(x/\log^A x).
 \end{equation}
 Here \begin{equation}c_1=\zeta(2)\zeta(3)/\zeta(6), \quad c_2=c_1\left(2\gamma-1-2\sum_{p}\frac{\log p}{p^2-p+1}\right).\end{equation}
Drappeau \cite[Thm 1.2]{Drappeau:2017} found improvements on the error, and showed that getting better estimates is related to the existence of Siegel zeroes. 

 \subsubsection{A variant of the Titchmarsh divisor problem}\label{sec:Titchmarsh}
We need a variant of the Titchmarsh divisor problem. 
Let \begin{equation}L_5(s, \chi_8)=(1+5^{-s})L(s,\chi_8)\end{equation} be the $L$-function related to $\chi_8$ with the Euler factor at $5$ removed, and let \begin{equation}C'=L_5(1,\chi_8)\prod_{p\neq 5}\left(1+\frac{\chi_8(p)}{p(p-1)}\right).\end{equation} 
Using \eqref{L-at-1}, we see that \begin{equation}
C'=\frac{6}{5}\frac{\log \varepsilon}{\sqrt{2}}\prod_{p\neq 5}\left(1+\frac{\chi_8(p)}{p(p-1)}\right).\end{equation} 
 \begin{thm} \label{Titchmarsh-type-asympt} Let $a=\pm 1\bmod 8$. Then for any $A>0$
 \begin{equation}  \sum_{\substack{n\leq x\\ n=a\bmod 8}}\mathcal N_2(5n+1)\Lambda(n)=\frac{C'}{\varphi(8)}x+O\left(\frac{x}{\log^A(x)}\right).\end{equation}
 \end{thm}
\begin{rem} Theorem \ref{Titchmarsh-type-asympt} is analogous to the Titchmarsh divisor problem in the following sense: The Titchmarsh divisor problem asks for asymptotics with error terms of $\sum_{n\leq x}1*1(n+1)\Lambda(n)$ and since  $\mathcal N_2(5n+1)=1*\chi_8(5n+1)$ the first expression in Theorem \ref{Titchmarsh-type-asympt} is an analogous sum over a linear shift with $n$ in an  arithmetic progression. Assing, Blomer and Li \cite{AssingBlomerLi:2021} studied such sums with $5$ replaced by $\pm 1$ and without the arithmetic progression. We use a variant of their method.  One can prove, using a slight variation of the proof given for Theorem \ref{Titchmarsh-type-asympt} that the same asymptotics hold, when $a=\pm 3 \bmod 8$ proving equidistribution among the four residue classes modulo $8$. 
\end{rem}
The proof of Theorem \ref{Titchmarsh-type-asympt} uses a recent result by Assing, Blomer and Li. Here we write $a| b^{\infty}$ to mean that $a$ has only prime divisors of $b$. 
\begin{thm}{(\cite[Thm 2.1]{AssingBlomerLi:2021})}\label{Assing-Blomer-Li-thm}
    There exist $0<\delta<1/2$ with the following property: Let 
    \begin{enumerate}
    \item $x\geq 2$ and $Q\leq x^{1/2+\delta}$, $A,C>0$, 
    \item $c,d\in \N$ with $d\vert c^\infty$, and $c,d\leq \log^C x$,
    \item $c_0,d_0\in \Z$ with $(c,c_0)=(d,d_0)=1$,
    \item $a_1,a_2\in \Z\backslash\{0\}$ with $\abs{a_1}\leq x^{1-\delta}$, $\abs{a_2}\leq x^\delta$.
    \end{enumerate} 
    Then
    \begin{equation}
    \sum_{\substack{q\leq Q\\ (q,a_1a_2)=1\\q=c_0\bmod c}}\Big(    \sum_{\substack{n\leq x\\ n=a_1\overline a_2\bmod q\\n=d_0\bmod d}}\Lambda(n)-\frac{x}{\varphi(qd)}\Big)=O_{A,C}\left(\frac{x}{\log^A(x)}\right).
    \end{equation}
\end{thm}
This theorem is particularly useful because the error term allows for varying $a_i, c,d$ in certain ranges. We have applied the prime number theorem and \cite[Lem. 5.1]{AssingBlomerLi:2021}) to get it in this form.

\begin{proof}[Proof of Theorem \ref{Titchmarsh-type-asympt}]
We first note that since $\Lambda(n)\leq \log(n)$ we have trivially \begin{equation}\label{one-trivial-bound}
\sum_{\substack{n\leq y\\ n=a\bmod 8}}\mathcal N_2(5n+1)\Lambda(n)=O(y \log y). 
\end{equation} Let $L=\log^B(x)$ for a suitably chosen $B$. It follows from \eqref{one-trivial-bound} that, up to an error of size $O_B(x/\log^{B-1}x)$, the sum in \eqref{Titchmarsh-type-asympt} equals 
\begin{equation}
\sum_{\substack{x/L < n\leq x\\ n=a\bmod 8}}\mathcal N_2(5n+1)\Lambda(n).
\end{equation}

    The arithmetic function $\mathcal N_2(m)$ is multiplicative and equal to $1$ on powers of $2$, so we may always remove the $2$-part of $m$. 
    
    If $n=1 \bmod 8$  then $5n+1$ is divisible by $2$ exactly once. 
    
To simplify notation we introduce the function
 \begin{equation}y(x)=\sqrt{(5x+1)/2}, \quad \hbox{with inverse } x(y)=\frac{2y^2-1}{5}.\end{equation}
 For simplicity we denote $y(n)$ by $y_n$.
    We can use Dirichlet's hyperbola method to get
\begin{align}
    \mathcal N_2(5n+1)&=\mathcal N_2(({5n+1})/{2})\\ &=\sum_{\substack{k\vert y_2^2\\ k< y_n }}\chi_8(k)(1+\chi_8(y_n^2))+\chi_8(y_n).
    \end{align}
Here we have set $\chi(y)=0$,  if $y\not\in \N$. Note that  $5n+1$ is never twice a square since two times a square is 0, 2, or 3 $\bmod$ 5.
It follows that 
    \begin{equation}\label{case-n=1}
    \mathcal N_2(5n+1)=\begin{cases}
        0, &\textrm{ if }n=1 \bmod 16,\\
 \displaystyle       2\sum_{\substack{k\vert y_n^2\\  k< y_n}}\chi_8(k),&\textrm{ if }n=9 \bmod 16.
    \end{cases} 
\end{equation}

We can now use these expressions to see that
\begin{align}
\sum_{\substack{x/L<n\leq x\\ n=1\bmod 8}}\mathcal N_2(5n+1)\Lambda(n)&=2\sum_{\substack{x/L<n\leq x\\ n=9\bmod 16}}\sum_{\substack{k\vert y_n^2\\  k< y_n}}\chi_8(k)\Lambda(n)
\\
&=2\sum_{\substack{k\leq y(x)}}\chi_8(k)\sum_{\substack{\max({x}/{L}, x(k))< n\leq x\\ n=9\bmod 16\\ 5n=-1 \bmod 2k}}\Lambda (n).\\
\intertext{When $k$ is divisible by $5$ the inner sum is void, and when $k$ is odd $n=9 \bmod 16$ implies $5n=-1 \bmod 2$. Therefore the two congruence conditions reduce to $n=9\bmod 16$ and $n=-\overline 5\bmod k$, and we get }
&=2\sum_{\substack{k\leq y(x)\\(n,5)=1}}\chi_8(k)\sum_{\substack{\max({x}/{L},x(k))< n\leq x\\ n=9\bmod 16\\ n=-\overline 5 \bmod k}}\Lambda (n)
\\ &=2(\Sigma_1+\Sigma_2).
\end{align} Here we have split the outer sum into two sums:  $\Sigma_1$ equals the sum over $k\leq y(x/L)$, and $\Sigma_2$ denotes the rest. We have 
\begin{align}
\Sigma_1&=\sum_{\substack{k\leq y(x/L)\\ (k,5)=1}}\chi_8(k)\sum_{\substack{{x}/{L}< n\leq x\\ n=9\bmod 16\\ n=-\overline 5 \bmod k}}\Lambda (n)\\
&=\sum_{b\bmod 8}\chi_8(b)\sum_{\substack{k\leq y(x/L)\\ (k,5)=1\\k=b\bmod 8}}\sum_{\substack{{x}/{L}< n\leq x\\ n=9\bmod 16\\ n=-\overline 5 \bmod k}}\Lambda (n)
\intertext{We can now apply Theorem \ref{Assing-Blomer-Li-thm} twice with $x$ equal to $x$ and $x/L$ respectively, both times with $Q=y(x/L)$, $a_1=-1$, $a_2=5$, $d_0=9$, $d=16$ $c_0=b$, $c=8$. This gives }
&=\sum_{b\bmod 8}\chi_8(b)\sum_{\substack{k\leq y(x/L)\\ (k,5)=1\\k=b\bmod 8}}\left(\frac{x}{\varphi(16k)}-\frac{x/L}{\varphi(16k)}\right)+O(x/\log^{A}(X))\\
&=\sum_{\substack{k\leq y(x/L)\\ (k,5)=1}}\chi_8(k)\left(\frac{x}{\varphi(16k)}-\frac{x/L}{\varphi(16k)}\right)+O(x/\log^{A}(x))\\
&=\frac{C'}{\varphi(16)}x+O(x/\log^A(x)),
\end{align}
where in the last line we have used \cite[Lem 5.2]{AssingBlomerLi:2021}.

We now want to show that $\Sigma_2\ll x/\log^A x$.  
\begin{align}
\Sigma_2&=\sum_{\substack{y(x/L)<k\leq y(x)\\(n,5)=1}}\chi_8(k)\sum_{\substack{x(k)< n\leq x\\ n=9\bmod 16\\ n=-\overline 5 \bmod k}}\Lambda (n).
\end{align}

We want to apply Theorem \ref{Assing-Blomer-Li-thm}, so we need to deal with the fact that the inner sum is over an interval depending on $k$. To address this we let 
\begin{equation}V=(1-\Delta) \textrm{ with } \Delta=\log^{-A/2} x,
\end{equation}
 and split the inner sum in intervals roughly of the form $yV<n\leq y$  as follows:
\begin{align}
\sum_{\substack{x(k)< n\leq x\\ n=9\bmod 16\\ n=-\overline 5 \bmod k}}\Lambda (n)&=\sum_{r=0}^{R(k,x)} \sum_{\substack{\max(x(k),xV^{r+1})< n\leq xV^r\\ n=9\bmod 16\\ n=-\overline 5 \bmod k}}\Lambda (n),\\
\intertext{where $R(k,x)=\min\{r\vert xV^{r+1}<x(k)\}$}
&=\sum_{r=0 }^{R(k,x)-1} \sum_{\substack{xV^{r+1}< n\leq xV^r\\ n=9\bmod 16\\ n=-\overline 5 \bmod k}}\Lambda (n)+ \sum_{\substack{x(k)< n\leq xV^{R(k,x)}\\ n=9\bmod 16\\ n=-\overline 5 \bmod k}}\Lambda (n). 
\end{align}
Observing that for  $k \in \Sigma_2$ we have $xV^{r+1}<x/L$ implies $xV^{r+1}<x(k)$. Therefore, $R(k,x)=O(\log L/\Delta)$ uniformly in $k$. 
Trivially we have 
\begin{align} 
x-x(k)=\sum_{r=0}^{R(k,x)-1} (xV^r-xV^{r+1}) +\left(xV^{R(x,k)}-x(k)\right),
\end{align}
so we arrive at
\begin{align}
\Sigma_2&=\sum_k\chi_8(k)\sum_{\substack{x(k)< n\leq x\\ n=9\bmod 16\\ n=-\overline 5 \bmod k}}\Lambda (n)=\sum_k\chi_8(k)\frac{x-x(k)}{\varphi(16k)}\\
&\quad +\sum_k\chi_8(k)\Big(\sum_{r=0}^{R(k,x)}\sum_{\substack{xV^{r+1}< n\leq xV^r\\ n=9\bmod 16\\ n=-\overline 5 \bmod k}}\Lambda (n)-\frac{xV^r-xV^{r+1}}{\varphi(16k)}\Big)\\
&\quad +\sum_k\chi_8(k)\Big(\sum_{\substack{x(k)< n\leq xV^{R(k,x)}\\ n=9\bmod 16\\ n=-\overline 5 \bmod k}}\Lambda (n)-\frac{xV^{R(k,x)}-x(k)}{\varphi(16k)}\Big)\\
&=\Sigma_{2,1}+\Sigma_{2,2}+\Sigma_{2,3}.
\end{align}
Here all $k$-sums are over $y(x/L)<k\leq y(x)$ with $(n,5)=1$.
Using summation by parts we see that $\Sigma_{2,1}=O(x/L)$.  The contribution from $\Sigma_{2,3}$ is bounded as follows: By Brun--Titchmarsh \eqref{Brun-Titchmarsh} and the definition of $R(k,x)$ the sum that comes after the character is bounded by $O(x\Delta/\varphi(k))$. Summing this over the relevant $k$, and using $\sum_{k\leq x}\varphi^{-1}(k)=O(\log x)$ (See \cite[Lemma 5.1]{AssingBlomerLi:2021}), gives a contribution of $\Sigma_{2,3}=O(x\log^{1-A/2}x).$

Finally, to estimate $\Sigma_{2,2}$, we first split the $k$ sum according to its value mod 8. Then we notice that as a function of $k$ with $y(x/L)<k\leq y(x)$ we have that $R(k,x)$ is decreasing. We then interchange the $k$ and the $r$ sum, and we find
\begin{align}
\Sigma_{2,2}&=\sum_{b\bmod 8}\chi_8(k) \sum_{\substack{k\\k=b\bmod 8}}\sum_{r=0}^{R(k,x)}\Sigma_{2,2}(n,k,r)\\
&=\sum_{b\bmod 8}\chi_8(k)\sum_{r=0}^{R(x)} \sum_{\substack{y(x/L)<k\leq K(x,r)\\(k,5)=1\\k=b\bmod 8}}\Sigma_{2,2}(n,k,r).
\end{align}
Here \begin{align}
\Sigma_{2,2}(n,k,r)&=\sum_{\substack{xzV^{r+1}< n\leq xV^r\\ n=9\bmod 16\\ n=-\overline 5 \bmod k}}\Lambda (n)-\frac{xV^r-xV^{r+1}}{\varphi(16k)},\\R(x)&=R(y(x/L),x)=O(\log L/\Delta),\\
K(x,r)&=\max\{k\vert r\leq R(k,x)\}.\end{align} We may now use Theorem \ref{Assing-Blomer-Li-thm} and we find that 
\begin{align}
    \Sigma_{2,2}=O\left(\sum_{b\bmod 8}\sum_{r=0}^{R(x)}\frac{xV^r}{\log^A(xV^r)}\right)=O\left(\frac{x}{\log^A(x/L)}\Delta\right)=O\left(\frac{x}{\log^{A(1/2-\varepsilon)}(x)}\right).
\end{align}
Here we have used that for $r$ in the sum we have $x/L\leq xV^r$. 
Summarizing we have shown, that for every $A>0$ we have 
\begin{equation}  \sum_{\substack{n\leq x\\ n=1\bmod 8}}\mathcal N_2(5n+1)\Lambda(n)=\frac{C'}{\varphi(8)}x+O\left(\frac{x}{\log^A(x)}\right).\end{equation} This proves the bound for $n=1\bmod 8$. 

To deal with the case of $n=7\bmod 8$  
 we note that in this case $5n+1$ is divisible by 2 exactly twice and the same procedure that lead to \eqref{case-n=1} leads to the following expression in this case:
    \begin{equation}
    \mathcal N_2(5n+1)=\begin{cases}\label{case-n=7}
        0, &\textrm{ if }n=15,23 \bmod 32,\\
 \displaystyle       2\sum_{\substack{k\vert w_n^2\\  k< w_n}}\chi_8(k)+\chi_8(w_n),&\textrm{ if }n=7, 31 \bmod 32.
    \end{cases} 
\end{equation}
    Here $w_n=\sqrt{((5n+1)/4 }$.
We then deal separately with the cases $a=7, 31$ of 
\begin{equation}  \sum_{\substack{n\leq x\\ n=a\bmod 32}}\mathcal N_2(5n+1)\Lambda(n).\end{equation}
The term $\chi_8(w_n)$ introduces a negligible error term, and the same technique which we employed above leads to 
\begin{equation}  \sum_{\substack{n\leq x\\ n=a\bmod 32}}\mathcal N_2(5n+1)\Lambda(n)=\frac{1}{2}\frac{C'}{\varphi(8)}x+O\left(\frac{x}{\log^A(x)}\right),\end{equation} when $a=7$ or $a=31$. Summing the contributions together we find 
\begin{equation}  \sum_{\substack{n\leq x\\ n=7\bmod 8}}\mathcal N_2(5n+1)\Lambda(n)=\frac{C'}{\varphi(8)}x+O\left(\frac{x}{\log^A(x)}\right),\end{equation} which finishes the proof. 

\end{proof}

\subsubsection{Counting over primes in the hyperbolic case}

We are now ready to prove the first part of Theorem \ref{main-theorem-hyphyp}. Recall that we are considering the sequence \begin{equation}\label{h-sequence}
h=(\psi(\g))=\left(\left(\frac{\log y_1}{2\log \e^2}, \frac{\log y_2}{2\log \e^2}\right)\right)\subseteq (\R\slash \Z)^2,\end{equation}
indexed over the set of all $\g\in H\backslash \G(2,5)\slash H$ with all four entries  strictly positive. Equip this index set with the size function $\nu(\g)=bc/5$, and recall from Proposition \ref{prop:basic-hyperbolic}, \eqref{prop:basic-hyperbolic-length} and \eqref{MMW-parameter-relation} that $\nu(\g)= (\cosh(2v(\g))-1)/10$.

\begin{thm}\label{we are getting there}
Consider the sequence $h$ in \eqref{h-sequence}. Then 
\begin{equation}\pi_h(x)=C\li(x)+O\left(\frac{x}{\log^Ax}\right),\end{equation}
where   \begin{equation}
    C=\frac{12}{5}\frac{\log \e}{\sqrt{2}}\prod_{p\neq 5}\left(1+\frac{\chi_8(p)}{p(p-1)}\right).\end{equation} Here $\chi_8$ is the even primitive Dirichlet character modulo 8.
\end{thm}

Note that by Lemma \ref{full-hyperbolic-decomposition} and the discussion before Theorem \ref{a cool map} the set we are indexing over corresponds precisely to $\pm H\Gamma H$  with $\pm \g\notin H$ and $\delta_1(\g)=\delta_2(\g)=0$. We can therefore parametrise this set using Theorem \ref{a cool map}.
Let $\g, \g'$ be the two matrices from \eqref{two-matrices}, which  map to the same element under $\psi$ in Theorem \ref{a cool map}. Observing that
\begin{equation}
    e\left(n_1\frac{\log y_1(\g')}{\log \varepsilon^4}+ n_2\frac{\log y_2(\g')}{\log\varepsilon^4}\right)=(-1)^{n_1+n_2}e\left(n_1\frac{\log y_1(\g)}{\log \varepsilon^4}+ n_2\frac{\log y_2(\g)}{\log\varepsilon^4}\right),
\end{equation}
 we see that, if $n>0$, then the Kloosterman type sum
\begin{equation}
     S_h(n_1,n_2,n)=\sum_{\substack{[\g]\\
\delta_i(\g)=0\\
(\cosh 2 v(\g)-1)/10=n}}e\left(n_1\frac{\log y_1(\g)}{\log \varepsilon^4}+ n_2\frac{\log y_2(\g)}{\log\varepsilon^4}\right)
\end{equation}
vanishes, unless $n_1, n_2$ have the same parity. If they do have the same parity, Theorem \ref{a cool map} gives 

\begin{align}\label{kloosterman-two circles}S_h(n_1,n_2,n)&=2\sum_{(z_1,z_2)\in D_K(5n+1)\times D_K(n)}\lambda(z_1)^{n_1+n_2}\lambda(z_2)^{n_1-n_2}\\
&=2U_{n_1+n_2}(5n+1)U_{n_1-n_2}(n).
\end{align}

It follows from Theorem \ref{goods-theorem} that if $m=(n_1,n_2)\in \Z^2$
then
\begin{align}
   \label{sum of hyp-hyp-kloosterman sums} \sum_{n\leq X}S_h(n_1,n_2,n)&=\sum_{\substack{[\g]
   \\
\delta_i(\g)=0\\
e^{2v(\g)}+O(1)\leq 20X+2 }}
\!\!\!\!\!\!
e\left(n_1\frac{\log y_1(\g)}{\log \varepsilon^4}+ n_2\frac{\log y_2(\g)}{\log \varepsilon}\right)\\ &=\delta_{m,0}\frac{(\log \varepsilon)^2}{\pi^2}10 X+ O_{m}(X^{2/3}).
\end{align}
The $n_1=n_2=0$ case reduces, via \eqref{kloosterman-two circles} and \eqref{counting ideals}, to 
\begin{align}
    \sum_{n\leq X}\mathcal N_2(5n+1)\mathcal N_2(n)&=\frac{5(\log \varepsilon)^2}{\pi^2} X+ O(X^{2/3}).
\end{align}
The same asymptotics were found by Hejhal \cite[Théorème 2]{Hejhal:1978}, \cite[Eq (1)]{Hejhal:1982b}, see also \cite[Théorème 8]{Hejhal:1982c}.

\begin{proof}[Proof of Theorem \ref{we are getting there}]
To investigate what happens if we only sum over primes in \eqref{sum of hyp-hyp-kloosterman sums} we define
\begin{equation}\psi_h(x)=\sum_{n\leq x}S_h(0,0,n)\Lambda(n),\end{equation}
where $\Lambda$ is the von Mangoldt function.
Since $S_h(n_1,n_2,n)\ll n^\e$ it is easy to verify that the asymptotical expansion of $\pi_h(x,0,0)$ is equivalent to \begin{equation}\label{equivalent-formulation}\psi_h(x)=C\cdot x+O(x/(\log(x))^A),\end{equation} which  we will prove. Since by \eqref{kloosterman-two circles} and \eqref{counting ideals},
\begin{equation}S_h(0,0,p)=2\mathcal N_2(p)\mathcal N_2(5p+1)=4\delta_{p=\pm 1  (8)}\mathcal N_2(5p+1)\end{equation} we have 
\begin{equation}\label{does-it-end}\psi_h(x)=4\sum_{\substack{n\leq x\\n=\pm 1\bmod 8}}\mathcal N_2(5n+1)\Lambda(n)+O(x/\log^A x).\end{equation}

It now follows from Theorem \ref{Titchmarsh-type-asympt} that
\begin{equation} \psi_h(x)=\frac{8}{\varphi(8)}cx+ O\left(\frac{x}{\log^A(x)}\right),\end{equation}
from which the claim follows.
\end{proof}

\subsection{Equidistribution over primes in the hyperbolic case}
We can now finally prove equidistribution over primes of the sequence $h$ in \eqref{h-sequence}. We first give upper bounds for the Weyl sums:
\begin{thm}\label{bounding-weyl-sums-hyperbolic}Consider the sequence $h$ in \eqref{h-sequence}, and consider integers $m_1,m_2$ not both zero. Then 
\begin{equation}
     \sum_{\substack{
\delta_i(\g)=0\\
\nu(\g)=p\leq x}}e\left(m_1\frac{\log y_1(\g)}{\log \varepsilon^4}+ m_2\frac{\log y_2(\g)}{\log\varepsilon^4}\right)=O\left(\frac{x}{\log x}\frac{1}{\log^{1-\frac{2}{\pi }}x}\right).
\end{equation}

\end{thm}
\begin{proof} Denoting the sum we want to bound by $B(m_1,m_2,x)$ we see that 
\begin{equation}
B(m_1,m_2,)=\sum_{p\leq x}S_h(m_1,m_2,p).
\end{equation}
Since $S_h(m_1,m_2,p)=0$ except when $m_1,m_2$ have the same parity, we may assume that $m_1,m_2$ have the same parity. We deduce from \eqref{kloosterman-two circles} that 
\begin{equation}
B(m_1,m_2,x)\leq 2\sum_{p\leq x}\abs{U_{n_1+n_2}(p)}\abs{U_{n_1-n_2}(5p+1)}.
\end{equation}
We now let $(a_1,b_1)=(1,0)$, $(a_2,b_2)=(5,1)$ and $g_1(n)=\abs{U_{m_1+m_2}(n)}$, $g_2(n)=\abs{U_{m_1-m_2}(n)}$. We verified in Section \ref{analysis-weyl-sums-hyperbolic} that the functions $g_i$ are multiplicative and satisfies $\abs{g_i(n)}\leq d(n)$. It now follows from Theorem \ref{nair-tenenbaum} that 
\begin{equation}
   \abs{ B(m_1,m_2,x)}\ll \frac{x}{\log^3(x)}\sum_{n_1\leq x}\frac{g_1(n_1)}{n_1}\sum_{n_2\leq x}\frac{g_2(n_2)}{n_2}.
\end{equation}
Since $g_i(n)\leq 1*\chi_8$ we have trivially $\sum_{n\leq x}g_i(n)/n=O(\log x)$. Since $m_1$, and $m_2$ are non-zero with the same parity  $m_1-m_2$, $m_1+m_2$ are even with at least one of them being non-zero. Let $g=g_i$ where $i$ is chosen such that $g_i=\abs{U_k}$ for some non-zero even $k$. We can then use Lemma \ref{powerful-bound} and summation by parts to conclude that 
\begin{equation}
\sum_{n\leq x}\frac{g(n)}{n}=O_k(\log^{2/\pi}(x)),
\end{equation}
which proves the claimed bound.
\end{proof}

Using Weyl's equidistribution criterion and Theorem \ref{we are getting there} we get the following corollary from Theorem \ref{bounding-weyl-sums-hyperbolic}:
\begin{cor}
    Consider the sequence $h$ in \eqref{h-sequence}. Then $h$ is equidistributed on   $(\R\backslash \Z)^2$ over primes.
\end{cor}

\bibliographystyle{amsplain}
\providecommand{\bysame}{\leavevmode\hbox to3em{\hrulefill}\thinspace}
\providecommand{\MR}{\relax\ifhmode\unskip\space\fi MR }
\providecommand{\MRhref}[2]{%
  \href{http://www.ams.org/mathscinet-getitem?mr=#1}{#2}
}
\providecommand{\href}[2]{#2}

\end{document}